\documentclass[hidelinks,onefignum,onetabnum]{siamart220329}
\usepackage{threeparttable}
\usepackage{amsmath}
\usepackage{graphicx}
\usepackage{float}
\usepackage{lipsum}
\usepackage{amsfonts}
\usepackage{graphicx}
\usepackage{epstopdf}
\usepackage{algorithmic}
\ifpdf
  \DeclareGraphicsExtensions{.eps,.pdf,.png,.jpg}
\else
  \DeclareGraphicsExtensions{.eps}
\fi

\newsiamremark{remark}{Remark}
\newsiamremark{hypothesis}{Hypothesis}
\crefname{hypothesis}{Hypothesis}{Hypotheses}
\newsiamthm{claim}{Claim}

\headers{Lagrangian cuts by batch}{Xiaoyu Luo, Chuanhou Gao}

\title{Lagrangian cuts generated by batch to efficiently solve two-stage stochastic mixed-integer program \thanks{Submitted to the editors DATE.
\funding{This work was funded by the National Nature Science Foundation of China under Grant No. 12320101001 and 12071428.}}}

\author{Xiaoyu Luo\thanks{School of Mathematical Sciences, Zhejiang University
  (\email{12135040@zju.edu.cn})}
\and Chuanhou Gao\thanks{School of Mathematical Sciences, Zhejiang University
  ( \email{gaochou@zju.edu.cn}).}
}

\usepackage{amsopn}

\ifpdf
\hypersetup{
  pdftitle={Lagrangian cuts generated by batch to efficiently solve tsSMIP},
  pdfauthor={Xiaoyu Luo, Chuanhou Gao}
}
\fi

\begin{document}

\maketitle

% REQUIRED
\begin{abstract}
We propose to generate Lagrangian cut for two-stage stochastic integer program by batch, in contrast to the existing methods which solve each Lagrangian subproblem at every iteration. We establish two convergence properties of the proposed algorithm. Then we demonstrate that the improvement in the lower bound achieved by incorporating the Lagrangian cut adheres to the `triangle inequality', thereby showcasing the superiority of our proposed method over existing approaches. Moreover, we suggest acquiring Lagrangian cuts for unresolved scenarios by averaging the coefficients of the acquired Lagrangian cuts, ensuring the quality of this cut with a certain probability. Computational study demonstrates that our proposed algorithm can significantly improve the lower bound of the linear relaxation of the Bender master problem more quickly with much fewer Lagrangian cuts.
\end{abstract}

% REQUIRED
\begin{keywords}
two-stage stochastic mixed-integer program, Lagrangian cut, batch
\end{keywords}

% REQUIRED
\begin{MSCcodes}
90C06, 90C11, 90C15
\end{MSCcodes}

\section{Introduction}
Two-stage stochastic mixed-integer program (tsSMIP) has attracted an increasing interest in recent decades due to its extensive applications in various problems, including the facility location problem \cite{muffak2023benders, duran2023efficient}, the network flow problem \cite{rahmaniani2018accelerating} and the logistics problem \cite{alfandari2022tailored}. Mathematically, tsSMIP, in the form of extensive formulation \cite{birge2011introduction}, can be modelled by 
\begin{subequations}\label{equa:SIP_ori}
\begin{align}
	     \min_{x,y^s}~&c^{\top} x+\sum_{s \in S} p_sd_s^{\top} y^s,  \\ 
		\text {s.t.}~~& A x=b,  \\
		&  T^s x+W^s y^s \geq h^s, ~\forall s \in S, \\
		& x \in \mathbb{R}_{+}^{n_1-p_1} \times \mathbb{Z}_{+}^{p_1},~ 
        y^s \in \mathbb{R}_{+}^{n_2}, ~\forall s \in S, 
	\end{align}
 \end{subequations}
where $c \in \mathbb{R}^{n_1}$ is the first-stage cost vector, $x$ represents the first-stage mixed-integer variable that should be determined before the stochastic scenarios reveal, $S = \left\{1,..., m\right\}$ is a finite scenario set with every scenario $s$ to occur randomly, $p_s$ the probability that scenario $s$ occurs, $d_s\in\mathbb{R}^{n_2}$ the second-stage cost vector, $y^s$ the second-stage decision variable for scenario $s$, $A \in \mathbb{R}^{m_1\times n_1}$ the first-stage constraint matrix, $b \in \mathbb{R}^{m_1}$, and $T^s \in \mathbb{R}^{m_2 \times n_1}$, $W^s \in \mathbb{R}^{m_2 \times n_2}$, $h^s \in \mathbb{R}^{m_2}$ are scenario-specific for each $s \in S$. It is usually to assume that the program is complete recourse to ensure the existence of solution $y$ for any given $x$. Clearly, when the number of scenarios is large, the magnitude of the problem will be prohibitively large, which makes it a great challenge to solve \cref{equa:SIP_ori}. To better solve it, the formulation of \cref{equa:SIP_ori} is usually rewritten as an alternative one
\begin{align}
\label{lagformulation}
    \min_{x, \theta_s}\left\{c^\top x + \sum_{s \in S}p_s\theta_s: (x, \theta_s) \in E^s, s\in S\right\},
\end{align}
where 
\begin{subequations}
\begin{align}
       &E^s := \left\{(x, \theta_s) \in X \times \mathbb{R}:  (x, y^s) \in K^s, \theta_s \geq d_s^\top y^s\right\}, \\
       &X:= \left\{x \in \mathbb{R}_{+}^{n_1-p_1} \times \mathbb{Z}_{+}^{p_1}: Ax = b\right\}, \\
       &K_s := \left\{x \in X, y: Ax \geq b, T^sx + W^sy^s \geq h^s , y^s \in \mathbb{R}_{+}^{n_2} \right\}. 
\end{align}
\end{subequations}
This reformulation allows to solve tsSMIP in separate steps, which suggests that for a fixed first-stage solution $x$, the minimum of the second-stage cost function at scenario $s$ can be computed by
\begin{subequations}\label{sub}
    \begin{align}
        f_s(x) := \min_{y_s}~&d_s^{\top}y^s,\\
        \text {s.t.}~&W^sy^s \geq h^s - T^sx,\\
        &y^s \in \mathbb{R}_{+}^{n_2}.
    \end{align}
\end{subequations}
Bender decomposition \cite{bnnobrs1962partitioning} precisely captures this characteristic and accordingly induces the so-called Bender formulation by approximating the lower bound of \cref{lagformulation} using linear inequalities, including the `Bender master problem'
\begin{subequations}\label{master}
    \begin{align}
        \min_{x, \theta_s}~&c^{\top}x + \sum_{s \in S}p_s\theta_s\\
        \text{s.t.}~& x \in X,\\
        &\theta_s \geq \lambda_s^{\top}(h^s - T^sx),~ s \in S, ~\lambda_s \in \text{Vert}(\Lambda_s) \label{optimality cut}
    \end{align}
\end{subequations}
and the `Bender subproblem' \cref{sub}. In \cref{master}, $\Lambda_s:= \left\{\lambda_s \in \mathbb{R}^{m_2}: \lambda_s^{\top}W^s\leq d_s^{\top}\right\}$ is the polyhedron associated with the dual of \cref{sub} and Vert$(\Lambda_s)$ represents the set of all the vertices of $\Lambda_s$. The constraint in \cref{optimality cut} is called the `Bender optimality cut'. Since the cardinality of constraint \cref{optimality cut} can be exponentially large, we usually use partial optimality cuts in \cref{optimality cut} to initialize the Bender decomposition and denotes such relaxation of \cref{master} by the relaxed master problem. The solution procedure consists of iteratively solving the master problem \cref{master} to get $x$ and the subproblem \cref{sub} to generate the Bender optimality until the optimal solution is reached.  

Despite being an alternative way to solve tsSMIP, the Bender decomposition suffers from several drawbacks, such as the slow convergence, the oscillation of the current relaxed optimal solution and the weak strength of the Bender cut. Thereamong, the weak strength is since the linear relaxation of \cref{master} can be viewed as the projection of the linear relaxation of \cref{equa:SIP_ori} onto the space of $(x, \theta_s)_{s \in S}$ through the solution of \cref{sub} \cite{bodur2017strengthened}, and the Bender optimality cut does not utilize the information from the integrality constraint in $X$. As to speak, the Bender cut is in fact rather weak in improving the lower bound of the Bender master problem \cref{master}. To enhance the quality of the generated Bender cut, the concept of `Pareto optimal' \cite{magnanti1981accelerating} was developed, which actually does not lift the Bender cut but only selects out the non-dominated one to accelerate the algorithmic convergence in case of multiple optimal solutions of \cref{sub}. The `cut-and-project' framework \cite{caroe1999decomposition, ntaimo2008computations, beier2015stage, bodur2017strengthened, zhang2014finitely} is another strategy built to lift the Bender cut. This class of methods add the valid inequalities derived from the integrality constraints to the second-stage subproblems \cref{sub}, and more tightened linear programs are created to generate stronger Benders cuts. Within this framework, the Gomory mixed-integer (GMI) cuts \cite{gade2014decomposition, zhang2014finitely} were conceptualized and exhibited some superiority. Further, Bodur et al. \cite{bodur2017strengthened} presented a theoretical outcome supporting the superiority of the `cut-and-project' framework, leveraging a heuristic method introduced by \cite{dash2010heuristic} to generate GMI cuts for \cref{sub} based on a specified first-stage solution. Following this research trajectory, Rahmanian et al. \cite{rahmaniani2020benders} introduced Benders dual decomposition (BDD), wherein Lagrangian cuts, derived by solving single-scenario MIPs, are generated and incorporated into the Benders formulation to enhance the relaxation. This decomposition can be regarded as an enhanced version of Bender decomposition, producing much more potent cuts but at the expense of solving a Mixed-Integer Linear Program (MILP). The acceleration to generate Lagrangian cuts was made by Chen and Luedtke \cite{chen2022generating}, who put forth techniques that involve in addressing the cut generation problem within a restricted subspace and employing a MIP approximation to discern a promising restricted subspace. 

Notwithstanding the above mentioned efforts, it is still quite time-consuming to generate the Lagrangian cut. In line with the work \cite{chen2022generating}, this paper also aims to improve the efficiency of generating Lagrangian cuts for solving tsSMIP with the strategy of `cut generation in batches'. Note that this strategy has been proved valid in generating Benders cuts \cite{blanchot2023benders} for solving two-stage stochastic linear programs through numerical experiments. The convergence of Bender decomposition is thus accelerated significantly. Balas et. al. \cite{balas1996mixed} systematically discussed this method used for general mixed-integer programs, but still on the level of numerical experiment validation. Based on these facts, we try to generate Lagrangian cuts in batches for tsSMIP, and thereby to enhance the efficiency of improving the lower bound of the Bender master problem of \cref{master}. Compared with the work in \cite{blanchot2023benders}, our algorithm tackles the situation of Lagrangian relaxation, and thus undergoes a more intricate convergence property analysis. In addition, we highlight a robust theoretical result that emphasizes the superiority of our proposed method. The main contributions of the current work may be summarized as
\begin{itemize}
    \item apply the `batch' strategy to generating Lagrangian cuts for accelerating solving tsSMIP;
    \item provide a theoretical support to say the effectiveness of the proposed Lagrangian cut generation algorithm, given by two convergence diagrams and a triangle inequality for bound improvement;
    \item utilize the information acquired from previously solved Lagrangian subproblems to generate averaged Lagrangian cuts;
    \item conduct extensive experiments on three classes of problems to display the efficiency of our algorithm in two types of separation methods.
\end{itemize}

The rest of the paper is organized as follows: \Cref{sec:pre} gives a brief introduction on Dual decomposition and the Lagrangian cut. This is followed by the algorithm development in \Cref{sec:batch} on generating Lagrangian cuts by batch, and some theoretic analyses are made towards proving its advantage. Further, \Cref{sec:average} contributes to generating the averaged Lagrangian cut and giving a probabilistic guarantee for its quality. In \Cref{sec:experiment} the efficacy of our algorithm is illustrated through comprehensive experiments. Finally, \Cref{conclusion} concludes the paper and also presents some points of future research.

\section{Preliminaries}\label{sec:pre}
In this section, the preliminaries about Dual decomposition and Lagrangian cut are given. 
\subsection{Dual decomposition}
Dual decomposition \cite{caroe1999dual} reformulates the problem \cref{equa:SIP_ori} by incorporating copies of the first-stage variables, and then creates
\begin{subequations}\label{DualD}
\begin{align}
    \min_{x,x^s,y^s}&\sum_{s\in S}p_s(c^\top x^s + d_s^\top y^s)\\
    \text{s.t.}~&Ax^s \geq b, s \in S,\\
    &T^sx^s + W^sy^s \geq  h^s, s\in S,\\
    &x^s \in X, y^s \in \mathbb{R}_{+}^{n_2}, s\in S,\\
    &x^s = x, s \in S.\label{nonanti}
\end{align}
\end{subequations}
By relaxing the constraint \cref{nonanti} with Lagrangian multipliers $\lambda_s \in \mathbb{R}^{n_1}$ for each $s \in S$, the Lagrangian relaxation problem of \cref{DualD} can be stated as
\begin{subequations}
\begin{align}
    z(\lambda) = \min_{x, x^s, y^s}~&\sum_{s \in S}p_s(c^\top x^s+ d_s^\top y^s) + \sum_{s \in S}p_s\lambda_s^\top(x^s - x),\\
    \text{s.t.}~~ &(x^s, y^s) \in K^s,~ s\in S.
\end{align}
\end{subequations}
The corresponding Lagrangian dual problem can be thus written to be
\begin{align}
    z = \max_{\lambda}\left\{z(\lambda): \sum_{s \in S}p_s \lambda_s = 0\right\},\label{lagrangian dual}
\end{align}
which induces the well-known equality
    \begin{align}
    \label{dualbound}
        z = \min_{x, y^s}\left\{c^\top x + \sum_{s \in S}p_sd_s^\top y^s: (x, y^s) \in conv(K^s), s \in S\right\}.
    \end{align}
The above equality depicts the tightness of the Lagrangian relaxation, by which the lower bound induced exhibits robust superiority through experiments \cite{schutz2009supply, solak2010optimization}.  

\subsection{Lagrangian cut}
Typically, it is difficult to address \cref{lagrangian dual}  since the inner problem encompasses multiple mixed-integer programs. To sidestep the direct solving \cref{lagrangian dual}, Bender dual decomposition \cite{rahmaniani2020benders} was proposed, which leverages \cref{lagformulation} and formulates a mixed-integer subproblem for each scenario to generate a singular Bender-type cut, termed by Lagrangian cut in the context. The definition is: $\forall (\pi,\pi_0) \in \mathbb{R}^n \times \mathbb{R}_+$, denote
\begin{align}
	\bar{Q}_s\left(\pi, \pi_0\right)&
    :=\min _{x, y}\left\{\pi^{\top} x+\pi_0d_s^{\top} y:(x, y) \in K^s\right\},\label{lagsub}
 \end{align}
 then
 \begin{align}\label{lagCut}
\pi^{\top} x+\pi_0 \theta_s \geq \bar{Q}_s\left(\pi, \pi_0\right)
  \end{align}
is called Lagrangian cut. We refer to it as $(\pi, \pi_0)$ in the following for convenience. 

Lagrangian cut is essentially a valid inequality for \cref{master}, which induces the separation problem to be
 \begin{align}
    h_{s}(\hat{x}) = \max_{(\pi,\pi_{0})\in \Pi_{s}}\left\{\Bar{Q}_{s}(\pi,\pi_0) - \pi^\top\hat{x} - \pi_0\hat{\theta}_{s}\right\}.\label{equa9}
\end{align}
Here, $\Pi_{s}$ can be any neighborhood of $\mathbb{R}^{n}\times \mathbb{R}^{+}$, and we call \cref{equa9} `Lagrangian subproblem'. Assume $(\hat{\pi}, \hat{\pi}_0)$ to be the optimal solution of \cref{equa9}, then 
\begin{align}
    \hat{\pi}^\top x + \hat{\pi}_0\theta_s \geq \Bar{Q}_s(\hat{\pi}, \hat{\pi}_0)
\end{align}
is a Lagrangian cut. The lower bound resulting from the inclusion of all the Lagrangian cuts can be expressed as \cite{chen2022generating}
\begin{align}
\label{lgcutfea}
    Z_{LD} := \min_{x,\theta}\left\{ c^{\top}x + \sum_{s\in S}p_s\theta_s: \pi^\top x + \pi_0\theta_s \geq \Bar{Q}_s(\pi, \pi_0), (\pi,\pi_0) \in \Pi_s, s\in S\right\}.
\end{align}
Chen and Luedtke \cite{chen2022generating} further claimed that \textit{the feasible region defined by all the Lagrangian cuts is equivalent to that of \cref{dualbound}, which renders $z = Z_{LD}$.} This claim indicates the role of Lagrangian cut in solving tsSMIP, where the enumeration process may be bypassed. It also provides a possibility of using Lagrangian cuts to approximate $conv(E^s)$. To accelerate this approximation convergence, Chen and Luedtke \cite{chen2022generating} designed the restricted separation algorithms, given in \Cref{appendix}. Like classic Bender decomposition, the true objective function of \cref{lgcutfea} is approximated by the cutting plane model
\begin{align}
\label{Lagrangian master}
    c^{\top}x + \sum_{s\in S}p_s\hat{Q}^k_s(x) = c^{\top}x + \sum_{s\in S}p_s\min_{x, \theta_s}\left\{\theta_s: \pi^{\top}x + \pi_0\theta_s \geq \gamma, (\pi, \pi_0, \gamma) \in \Phi^k_s   \right\},
\end{align}
where $\Phi^k_s$ represents the Bender cuts and Lagrangian cuts that have been added to the master problem up to iteration $k$.

\section{Lagrangian cut generation algorithm and theoretical support}
\label{sec:batch}
In this section, we develop the algorithm to generate Lagrangian cut by batch, and then give some theoretic analysis to ensure efficacy. 
\subsection{Algorithm design}
As said in Algorithm A.1, when a Lagrangian cut is generated for each scenario $s \in S$, it needs to solve a bi-level program \cref{equa9}; and when a round of Lagrangian cuts are generated (i.e., all scenarios are traversed), the Lagrangian master problem \cref{Lagrangian master} is updated and further solved \cite{chen2022generating, rahmaniani2020benders}. The whole process still keeps time-consuming. To further improve the efficiency of solving, we borrow the `batch' strategy to generate Lagrangian cut. Specifically, at each iteration, we solve the Lagrangian subproblem \cref{equa9} by batch, and then go back to resolve the master problem as soon as some stopping condition is attained. This can be done by checking the corresponding objective value of the Lagrangian subproblem to judge if the total violation exceeds a preset threshold. Based on the work in \cite{chen2022generating}, the cut with too small $\pi_0$ has little impact on the lower bound of the master problem \cref{master}, so we do not consider the cut with coefficient $(\pi, \pi_0)$ in the whole space but only those with coefficient like $(\pi, 1)$, where $\pi$ belongs to a compact set $\Pi_s$ of $\mathbb{R}^n$. Therefore, the lower bound achieved by incorporating all of these Lagrangian cuts can be denoted as
\begin{subequations}
\label{complete}
\begin{align}
    \min_{x, \theta_s}~&c^{\top}x + \sum_{s \in S}p_sQ(x, s),\label{p1}\\
    \text{s.t.}~& x \in \Bar{X}:= \left\{x \in \mathbb{R}_{+}^{n_1}: Ax = b\right\}, \label{p1_2}
\end{align}
\end{subequations}
where 
\begin{align}
    Q(x,s) := \min~&\theta_s \nonumber \\
    \text{s.t.}~&\pi^\top x + \theta_s \geq \Bar{Q}_s(\pi, 1), ~\pi \in \Pi_s. \nonumber 
\end{align}

In the following, we develop the algorithm to generate Lagrangian cuts to approximate \cref{complete}, described in \Cref{a1}. In the algorithm, $\tau = \lceil m/\kappa \rceil$ represents the number of the batches; the scenario set $S$ is divided into $S = \cup_{i = 1}^\tau P_i$, where $P_i$ is a batch consisting of $\kappa$ scenarios and $P_i \cap P_j = \emptyset $ $\forall i, j \in [1,...,\tau ]$; $(RMP)_k$ is called the relaxed master program, representing the relaxation of \cref{complete} with partial Lagrangian cuts at iteration $k$. As soon as $(RMP)_k$ is solved to yield a relaxed optimal solution $(\hat{x}^k, \hat{\theta}_s^k)_{s \in S}$, we need to arrange the order of the batches to solve subproblems in turn. Further, we define the concept of `$\epsilon$-optimal solution' to associate with the `stopping criterion' used in \Cref{a1}.

\begin{algorithm}
\caption{Generating Lagrangian cut by batch}
\begin{algorithmic}[1]\label{a1}
\STATE Parameters: $\epsilon \geq 0$, $\kappa$ the batch size, $\tau$ the number of the batches.
\STATE $t \leftarrow 1$, $k \leftarrow 0$, stopping criterion $\leftarrow$ False
\STATE Obtain the scenario batch $S = \left\{P_1,..., P_{\tau}\right\}$ according to the given batch size.
\WHILE{$t < \tau + 1$}
\STATE $k \leftarrow k+1$
\STATE Solve $(RMP)_k$ and obtain the current relaxed optimal solution ($\hat{x}_k, \hat{\theta}^k_s)_{s \in S}$
\STATE $t \leftarrow 1$, stopping criterion $\leftarrow$ False
\STATE Choose a permutation $\sigma$ to order these batches
\WHILE{stopping criterion = False and $t < \tau +1$}
\STATE $t=t+1$
\FOR{$ s \in P_{\sigma(t)}$}
\STATE Solve the Lagrangian subproblem $h_s(\hat{x}_k)$, obtain the Lagrangian cut $(\pi_s,1)$ and add the cut to $(RMP)_k$.
\ENDFOR
\IF{$\sum_{s \in \cup_{i = 1}^{t}P_{\sigma(i)}}p_s(\Bar{Q}_s(\pi_s, 1) - \pi_s^{\top} \hat{x}_k - \hat{\theta_s^k}) \leq \epsilon$}
\STATE Continue
\ELSE
\STATE stopping criterion $\leftarrow$ True
\ENDIF
\ENDWHILE
\STATE $(RMP)_{k+1} \leftarrow (RMP)_{k}$
\ENDWHILE 
\STATE Return the $\epsilon$-optimal solution $\hat{x}_k$
\end{algorithmic}
\end{algorithm}

\begin{definition}[$\epsilon$-optimal solution]
Denote the lower bound of $(RMP)_k$ by $l_b^k = c^\top\hat{x}^k+ \sum_{s\in \mathcal{S}}p_s\hat{\theta}_s^k$ and the current objective value of \cref{complete} at the first-stage solution $x$ by $u_b(x) = c^\top x + \sum_{s \in \mathcal{S}}{p_s Q(x,s)}$. If the Lagrangian cut is not separated exactly, $u_b(x)=c^\top x + \sum_{s \in \mathcal{S}}p_s(\Bar{Q}_s(\pi_s, 1) - \pi_s^{\top} x)$, where $(\pi_s, 1)_{s \in \mathcal{S}}$ are the generated Lagrangian cuts. Then for any optimality gap $\epsilon \geq 0$, the first-stage solution $\hat{x}_k$ at iteration $k$ is not $\epsilon$-optimal if $u_b(\hat{x}_k) - l_b^k > \epsilon$; otherwise, it is $\epsilon$-optimal.\label{optimal}
\end{definition}

%The definition suggests a criterion to stop solving the Lagrangian subproblems. That is to say, if at  some iteration $k$ and some batch $t$ the first-stage solution $\hat{x}_k$ is not $\epsilon$-optimal, then the solving of the Lagrangian subproblems will stop, and the update to $(RMP)_k$ will start up.   

%Let $\epsilon \geq 0$ be the optimality we set, we then define when the current first-stage solution $x \in X$ can be cut off. 
%Here we give a new criteria to check whether the $\hat{x}_k$ is $\epsilon$-optimal in order to update $(RMP)_k$ in advance, avoiding to solve all the Lagrangian subproblems.\

%The first-stage solution $\hat{x}_k$ is not $\epsilon$-optimal means the generated Lagrangian cuts at this point can cut off $\hat{x}_k$ to some degree. 
%The method proposed in \cite{chen2022generating} and \cite{rahmaniani2020benders} requires to solve all the Lagrangian subproblems at each iteration. \
Based on \Cref{optimal}, we define the stopping criterion below.
\begin{definition}[stopping criterion]
As given in Line 14 in \Cref{a1}, the condition of $\sum_{s \in \cup_{i = 1}^{t}P_i}p_s(\Bar{Q}_s(\pi_s, 1) - \pi_s^{\top} \hat{x}_k - \hat{\theta_s^k}) > \epsilon$ is said to be a stopping criterion for $\hat{x}_k$ at some batch $t$ $(t\leq \tau)$ during the $k$th iteration process.  \label{stopping}
\end{definition}

Utilizing the stopping criterion, $\Cref{a1}$ can stop solving the Lagrangian subproblems of scenarios beyond batch $t$ ($t<\tau$), and jump out of the loop to resolve $(RMP)_k$, which will avoid to solve all the Lagrangian subproblems about every scenario batch before updating $(RMP)_k$. As a result, the time may be saved greatly using the current regime compared to that used to solve all the Lagrangian subproblems \cite{chen2022generating}. The following proposition reports a necessary and sufficient condition to suggest an $\epsilon$-optimal solution.

\begin{proposition}
For $(RMP)_k$ at iteration $k$, the solution $\hat{x}_k$ is $\epsilon$-optimal if and only if no scenario batch can trigger the `stopping criterion'.
\end{proposition}

\begin{proof}
    ($\Rightarrow$) If $\hat{x}_k$ is $\epsilon$-optimal, $\epsilon \geq u_b(\hat{x}_k) - l_b^k = \sum_{s \in \mathcal{S}}{p_s Q(\hat{x}_k,s)} - \sum_{s\in \mathcal{S}}p_s\hat{\theta}_s^k = \sum_{s \in S}p_s(\Bar{Q}_s(\pi_s, 1) - \pi_s^{\top} \hat{x}_k - \hat{\theta_s^k}) \geq \sum_{s \in \cup_{i = 1}^{t}P_i}p_s(\Bar{Q}_s(\pi_s, 1) - \pi_s^{\top} \hat{x}_k - \hat{\theta_s^k})$ for any $t \in [1, \tau]$. Therefore, no scenario batch can trigger the `stopping criterion'.\

    ($\Leftarrow$) If no scenario batch can trigger the `stopping criterion', the proof is straightforward since each of the above procedures is reversed.
\end{proof}

%\begin{proposition}
%At iteration k, the $\hat{x}_k$ is $\epsilon$-optimal if and only if no batch can reach the stopping conditions.
%\end{proposition}

%It is obvious to see that at each iteration, either the first-stage variable $\hat{x}_k$ is proven to be $\epsilon$-optimal or the algorithm continues with the generated Lagrangian cuts added to $(RMP)_k$. The subsequent definition of permutation provides a means to represent the order in which these batches are to be solved.

The concept of permutation in Line 8 in \Cref{a1} is defined as follows.  
\begin{definition}[permutation]
A permutation $\sigma$ is a bijection mapping $\left\{1,..., \tau\right\}$ to itself. \label{permutation}
\end{definition}

We give a simple example to exhibit permutation: $\sigma = 
\left(
\begin{array}{llll}
1 &2 &3~...&\tau\\
\tau&1&2~...&\tau -1
\end{array}
\right)
$
is a permutation that maps 1 to $\tau$, 2 to 1, and so on. In \Cref{a1}, the permutation is used to represent the order of scenario batches to be solved. The permutation  $\left\{P_{\sigma(1)}, P_{\sigma(2)},..., P_{\sigma(\tau)}\right\}$ emerging in Line 11 means the order to be solved from left to right.

Finally, we provide a bird's eye-view of \Cref{a1}. The whole loop from Line 4 to Line 21 indicates a complete loop for the master problem $(RMP)_k$, a flow of which serves to update the mater problem and yield a new relaxed optimal solution. At the beginning of each loop, we retrieve a new first-stage solution and choose a new permutation to reorder scenario batches to be solved. The while loop from Lines 9 to 19 indicates a loop of generating Lagrangian cuts for $(RMP)_k$. In this loop, Lines 11 to 13 show the process of generating Lagrangian cuts in a chosen batch and Lines 14 to 18 are used for checking whether the generated Lagrangian cuts are violated greatly (referred to the preset threshold) by current relaxed optimal solution. If the threshold is attained, we update $(RMP)_k$ and enter into the next loop (going back to Line 4), otherwise this algorithm returns an $\epsilon$-optimal solution.

\subsection{Theoretical analysis}
In this subsection, we make some theoretical analysis on the convergence of the result suggested by \Cref{a1} to the optimal solution, discussed through two cases of $\epsilon>0$ and $\epsilon=0$. 

\begin{lemma}
\label{lip}
The function $\Bar{Q}_s(\pi, 1)$ defined in \cref{lagsub} is Lipschitz continuous with respect to variable $\pi$.
\end{lemma}

\begin{proof}
From the definition of $\Bar{Q}_s(\pi, 1)$, we have 
$$\bar{Q}_s\left(\pi, 1\right)
    :=\min _{x, \theta^s}\left\{\pi^{\top} x + \theta_s:(x, \theta^s) \in E_s\right\}.$$
Since this linear programming is sure to have an optimal solution and the extreme point of $conv(E_s)$ is finite, $\Bar{Q}_s(\pi, 1)$ can be viewed as the lower bound of several linear functions. Therefore, $\Bar{Q}_s(\pi, 1)$ is a piecewise linear function, which is Lipschitz continuous.
\end{proof}

\begin{theorem}
\label{convergence} For \Cref{a1}, (1) in the case of $\epsilon > 0$, the output of the algorithm converges to an $\epsilon$-optimal solution of the problem \cref{complete} in a finite number of iteration steps; 
(2) in the case of $\epsilon = 0$, the accumulation point of the sequence $\left\{(\hat{x}_k, \hat{\theta}_s^k)_{s \in S}\right\}_{k = 1}^{\infty}$ generated by $(RMP)_k$ is optimal to the problem \cref{complete}.
\end{theorem}

\begin{proof}
(case 1: $\epsilon > 0 $) It is obvious to see that if the algorithm terminates in a finite number of steps, the  solution returned by \Cref{a1} is an $\epsilon$-optimal solution. Assume that this algorithm can not terminate in finite steps. Then we can generate a sequence $\left\{(\hat{x}_k, \hat{\theta}_s^k)_{s \in S}\right\}_{k = 1}^{\infty}$ and a corresponding sequence of Lagrangian cuts $\left\{(\pi^k_{s},1)_{s \in S}\right\}_{k = 1}^{\infty}$, where $(\pi^k_{s},1)_{s \in S}$ cut off $(\hat{x}_k, \hat{\theta}_s^k)_{s \in S}$. Note that at each iteration, only part of the scenario set generates Lagrangian cuts, therefore we can design the $(\pi^k_{s},1)$ for the other part as a cut that has been added and is active at point $(\hat{x}_k, \hat{\theta}^k_s)$. For any $k_1 > k_2$, 
\begin{align}
\sum_{s \in S}p^s[\Bar{Q}_{s}(\pi^{k_2}_{s},1) - (\pi^{k_2}_{s})^{\top}\hat{x}^{k_{1}} -\hat{\theta}_s^{k_1}] \leq 0,\label{1}
\end{align}
because $(\hat{x}_{k_1}, \hat{\theta}_{s}^{k_1})_{s \in S}$ can not violate the Lagrangian cuts that have been generated.

On the other hand, because $(\hat{x}_{k_1}, \hat{\theta}_s^{k_1})_{s \in S}$ violates the Lagrangian cut defined by $(\pi^{k_1}_{s},1)_{s \in S}$ at least by $\epsilon$: 
\begin{equation}
\sum_{s \in S}p^s[\Bar{Q}_{s}(\pi^{k_1}_{s}, 1) - (\pi^{k_1}_{s})^{\top}\hat{x}_{k_{1}} -\hat{\theta}_s^{k_1}] \geq \epsilon.\label{2}
\end{equation}
Because the sequence $\left\{(\hat{x}_k,\hat{\theta}_s^k)_{s \in S}\right\}_{t = 1}^{\infty}$ is bounded and $\Bar{Q}_{s}(\pi, 1)$ is  Lipschitz continuous, there exists a positive number $\gamma$ such that $\lvert(\pi^{k_1}_{s}, 1) - (\pi^{k_2}_{s}, 1)\rvert \geq \gamma$ for all $s \in S$. Otherwise the gap between \cref{1} and \cref{2} can not exceed $\epsilon$. However, this contradicts with the fact that $(\Pi_s)_{s \in S}$ is a compact set.

(case 2: $\epsilon = 0$) In this case, we can not infer the inequality \cref{2} directly, because the violation at each iteration is not necessarily larger than some fixed positive number. We assume that the accumulation point $(\hat{x}, \hat{\theta}_s)_{s \in \mathcal{S}}$ is not optimal, then we have the following inequality: $\sum_{s \in \mathcal{S}}p^s\max \left\{\Bar{Q}_s(\pi_s, 1 ) - \pi_s^{\top} \hat{x} - \hat{\theta}_s\right\} = \sigma > 0$. The intuition is easy: when the points in the sequence is close enough to the point $(\hat{x}, \hat{\theta}_s)_{s \in \mathcal{S}}$, they can also be cut off with a violation stricly larger than zero. Indeed:
{\small
\begin{align}
&\sum_{s \in \mathcal{S}}(\max_{\pi}\left\{\Bar{Q}_s(\pi_s, 1) - \pi_s^{\top}\hat{x} - \hat{\theta}_s: \pi_s \in \Pi_s\right\} - \max_{\pi_s}\left\{\lvert\pi_s^{\top}(\hat{x}_k - \hat{x}) + \hat{\theta}^k_s - \hat{\theta}_s \rvert: \pi_s \in \Pi_s \right\})\nonumber\\
         &\leq\sum_{s \in \mathcal{S}}\max_{\pi_s}\left\{\Bar{Q}_s(\pi_s, 1) - \pi_s^{\top}\hat{x}_k - \hat{\theta}_s^k: \pi_s \in \Pi_s\right\}\nonumber
\end{align}
}
holds for every $k \in \mathcal{Z}$.
Because $\left\{(\hat{x}_k, \hat{\theta}_s^k)_{s \in S}\right\}_{t = 1}^{\infty}$ converges to  $(\hat{x}, \hat{\theta}_s)_{s \in \mathcal{S}}$, there exists an positive integer number $N$, such that
$\forall k \geq N$, 
$$
\sum_{s \in \mathcal{S}}\max_{\pi_s}\left\{\lvert\pi_s^{\top}(\hat{x}_k - \hat{x}) + \hat{\theta}^k_s - \hat{\theta}_s \rvert : \pi_s \in \Pi_s \right\} \leq \sigma/2.
$$
Then we can obtain a similar inequality to \cref{2}. The remaining proof is the same as case (1).
\end{proof}

\begin{corollary}
\label{conver col}
In the case of $\epsilon > 0$, if the Lagrangian cut is generated with a tolerance $\delta > 0$, then the output of \Cref{a1} converges to an $\frac{\epsilon}{1-\delta}$-optimal solution in a finite number of iteration steps.
\end{corollary}

\begin{proof}
If the algorithm terminates in a finite number of steps with a solution $(\hat{x}_k,\hat{\theta}_s^k)_{s \in S}$, then 
$$
\begin{aligned}
    &\sum_{s \in S}(1-\delta)p^s\max_{(\pi_s, 1)}\left\{\Bar{Q}_{s}(\pi_s, 1) - \pi_s^{\top}\hat{x}_k - \hat{\theta}_s^k, \pi_s \in \Pi_{s}\right\}\\
    \leq &\sum_{s \in S}p^s\left\{\Bar{Q}_{s}(\pi^k_s,1) - (\pi_s^k)^{\top}\hat{x}_k - \hat{\theta}_s^k, \pi_s^k \in \Pi_{s}\right\} \\
    \leq& \epsilon
\end{aligned}
$$
Therefore, $ \sum_{s \in S}p^s\max_{\pi_s}\left\{\Bar{Q}_{s}(\pi_s, 1) - \pi_s^{\top}\hat{x}_k - \hat{\theta}_s^k, \pi_s \in \Pi_{s}\right\} \leq \frac{\epsilon}{1-\delta} $. The proof of convergence in finite number of steps is similar to the one in \Cref{convergence}.
\end{proof}

\begin{remark}
 \Cref{convergence} articulates the convergence of \Cref{a1} about generating Lagrangian cuts by batch. \Cref{conver col} further depicts the extent that the algorithm output can tolerate  the cut generation gap. Clearly, even if the tolerance $\delta$ reaches $50\%$, the ultimate returned first-stage solution is still within a gap of $2\epsilon$.
 \end{remark}

Besides the convergence, \cref{a1} also exhibits a potential of efficiency due to a large possibility of only solving Lagrangian subproblems attributing to part scenarios during every iteration. This simultaneously means that fewer Lagrangian cuts need to be added to the corresponding master problem for the same lower bound improvement. We continue to elucidate this point from the theoretical level through beginning with a simple case.  

Consider a Bender-type master formulation of a tsSMIP with two scenarios, labeled by $sce_1$ and $sce_2$, respectively. 
\begin{subequations}\label{t1}
\begin{align}
    \min_{x, \theta}~&c^{\top}x + \theta_1 + \theta_2 \label{3},\\
    \text{s.t.}~&\theta_1 e \geq D_1 x + H_1, \label{4}\\
    &\theta_2 e \geq D_2 x + H_2, \label{5}\\
    &x \in X, \theta_1 \in \mathbb{R}, \theta_2 \in \mathbb{R},\label{3.8}   
\end{align}
\end{subequations}
where $e$ is a vector comprised entirely of ones, and (\ref{4}) (\ref{5}) are the already added Bender-like cuts for two scenarios, respectively. Then we try to add Lagrangian cuts generated by single scenario. For the needs of theoretical analysis, the candidate Lagrangian cuts generated by solving the Lagrangian subproblems \cref{equa9} are assumed as those that can improve the bound to the greatest extent, denoted individually by  
\begin{subequations}
\begin{align}
&\theta_1 + \pi_1^{\top} x \geq \Bar{Q}_1(\pi_1, 1)\label{c1},\\
&\theta_2 + \pi_2^{\top} x \geq \Bar{Q}_2(\pi_2, 1).\label{c2}
\end{align}
\end{subequations}
The coefficients $\pi_i (i=1,2)$ should satisfy
%We generate Lagrangian cuts with respect to the two scenarios in the sense of the most bound improvement. (Note that here is not  to generate Lagrangian cut  in order to analyse theoretically ). Firstly, we show the notations for generating Lagrangian cut for each scenario at a iteration (i.e. by no batch):
\begin{subequations}\label{pi-single}
\begin{align}
    \pi_i = \mathop{\arg \max}_{\pi} \min_{x, \theta}~&c^{\top}x + \theta_1 + \theta_2,\label{7} \\
    \text{s.t.}~&\theta_1 e \geq D_1 x + H_1, \\
    &\theta_2 e \geq D_2 x + H_2, \\
    &\theta_i + \pi^{\top} x \geq \Bar{Q}_i(\pi, 1),\\
    &x \in X, \theta_1 \in \mathbb{R}, \theta_2 \in \mathbb{R}, i = 1,2. \label{7.1}
\end{align}
\end{subequations}
The enhanced lower bound resulting from the incorporation of these two cuts into the master problem \cref{t1} is denoted as $d(sce_1, sce_2)$.

Next, we consider Lagrangian cuts generated by batch (one batch contains $sce_1$ and the other contains $sce_2$), which follow
\begin{subequations}\label{pi1-batch}
\begin{align}
    \pi_1 = \mathop{\arg \max}_{\pi} \min_{x, \theta}~&c^{\top}x + \theta_1 + \theta_2, \label{12}\\
    \text{s.t.}~&\theta_1 e \geq D_1 x + H_1, \\
    &\theta_2 e \geq D_2 x + H_2, \\
    &\theta_1 + \pi^{\top} x \geq \Bar{Q}_1(\pi, 1),\\
    &x \in X, \theta_1 \in \mathbb{R}, \theta_2 \in \mathbb{R}
\end{align}
\end{subequations}
and 
\begin{subequations}\label{pi2-batch}
\begin{align}
    \pi_2^* = \mathop{\arg \max}_{\pi} \min_{x, \theta}~&c^{\top}x + \theta_1 + \theta_2, \\
    \text{s.t.}~&\theta_1 e \geq D_1 x + H_1, \label{18}\\
    &\theta_2 e \geq D_2 x + H_2, \\
    &\theta_1 + \pi_1^{\top} x \geq \Bar{Q}_{1}(\pi_1, 1),\\
    &\theta_2 + \pi^{\top} x \geq \Bar{Q}_2(\pi, 1),\\
    &x \in X, \theta_1 \in \mathbb{R}, \theta_2 \in \mathbb{R}.\label{22}
\end{align}
\end{subequations}
Note that for $sce_1$, the coefficient $\pi_1$ is completely the same in two situations (\cref{pi-single} and \cref{pi1-batch}) while for $sce_2$ $\pi_2$ and $\pi_2^*$ are different. Their relation mainly relies on the order of Lagrangian subproblems to be solved. Although different order may result in different lower bound improvement, it will not affect the subsequent result. By denoting the lower bound improvement in the `batch' situation by $d(sce_1) + d(sce_2)$, we have
\begin{theorem}
\label{thm-triangle}
The lower bounds improved by two kinds of Lagrangian cuts given in \cref{pi-single} and \cref{pi1-batch} plus \cref{pi2-batch} satisfy the triangle inequality 
\begin{align}
d(sce_1) + d(sce_2) \geq d(sce_1, sce_2).\label{ine1}
\end{align}
\end{theorem}

\begin{proof}
We denote the optimal solution of (\ref{12}) by $x^*$ and consider two cases:

(1) The cut $\theta_2 + \pi_2^{\top} x \geq \Bar{Q}_2(\pi_2, 1)$ can not cut off the first-stage solution $x^*$, then this cut has no contribution to the lower bound improvement. Inequality \cref{ine1} holds obviously for this case.

(2) The cut $\theta_2 + \pi_2^{\top} x \geq \Bar{Q}_2(\pi_2, 1)$ can cut off the first-stage solution $x^*$, then we have 
\begin{align}
&\left\{\max_{\pi} \min_{x, \theta} c^{\top}x + \theta_1 + \theta_2, \cref{pi2-batch}(b\sim f)\right\} \geq \nonumber\\
&\left\{\min_{x, \theta} c^{\top}x + \theta_1 + \theta_2 , \cref{t1}(b\sim d), 
\cref{c1}, \cref{c2}\right\}.\nonumber
\end{align}
Therefore, the inequality \cref{ine1} also holds.
\end{proof}

\begin{remark}
\Cref{thm-triangle} means that compared with the existing method developed in \cite{chen2022generating}, the current \Cref{a1} can make the master problem updated more frequently and can fully utilize the role of each Lagrangian cut. Therefore, we can get larger lower bound improvement by generating fewer cuts.
\end{remark}

Further, we extend the result to the general one through replacing $sce_1$ and $sce_2$ by 
two batches $B_1 = \left\{s^1_i\right\}_{i = 1}^{m_{b1}}$ and $B_2 = \left\{s^2_j\right\}_{j = 1}^{m_{b2}}$, respectively.
$d(B_1) + d(B_2)$ and $d(B_1, B_2)$ individually represents the lower bound improved by generating Lagrangian cut for scenarios in $B_1$ and $B_2$ in two consecutive iterations and in a single iteration. Similar to \Cref{thm-triangle}, we get the following triangle inequality
\begin{corollary}\label{corro:ineq}
  $d(B_1) + d(B_2) \geq d(B_1, B_2).$
\end{corollary}

\begin{proof}
The proof of is similar to that of \Cref{thm-triangle}.
\end{proof}

It should be mentioned that the the above triangle inequalities are derived from adding Lagrangian cuts to improve the lower bound of the corresponding master problem. Intuitively, these results may be applied to any general cutting plane method in MIP, exploring the delicate balance between the number of cuts added in an iteration and the number of iterations. Although this idea  has actually been mentioned in \cite{balas1996mixed}, the experimental results in that paper did not demonstrate any advantage in generating cuts by batch. Nevertheless, given the advancements in optimization solvers for linear programs over time, it becomes intriguing to investigate the conditions under which generating general cutting planes in MIP by batch, such as Gomory cut, can exhibit superiority.\

%Furthermore, \Cref{theorem} only considers one paradigm for generating Lagrangian cut, that is we generate one Lagrangian cut for each scenario separately. In other words, program \cref{7}-\cref{7.1} searches for the Lagrangian cut for one scenario s that improves the lower bound of the current master problem instead of considering all scenarios jointly. It is obvious that the latter paradigm can obtain a tighter lower bound and we will present a similar theory for this paradigm in \Cref{sec:appen}.

\section{Averaged Lagrangian cut}
\label{sec:average}
In this section, we will define `averaged Lagrangian cut' as additional information to accelerate the convergence of \Cref{a1}. 

As can be seen from the generation process of Lagrangian cut, it needs to solve a bi-level program, unlike generating Bender cut, where only a linear program needs to be solved. It is reasonable to believe that the former may provide additional but valuable information to be further utilized. A naive idea is to utilize them to generate new valid inequalities as Lagrangian cuts of the remaining scenarios (unsolved subproblems), which may help to accelerate the convergence of \Cref{a1}. The process sounds like a machine learning process, where the generated Lagrangian cuts for some scenarios are obtained through training while the unsolved Lagrangian cuts for the other scenrios are obtained through testing the trained result. This conversely implies that machine learning may be a potential way to learn Lagrangian cut in the future study. Motivated by the work in \cite{bertsimas2023stochastic}, we try to acquire new Lagrangian cuts through averaging the known ones. To this task, we give the definition of cut strength first.

% we assume that the first-stage solution resides within the relative interior point of $\cap_{s \in \mathcal{S}}proj_{x}conv(K_s)$. 
 
% The core concept involves averaging the coefficients of these obtained Lagrangian cuts to form a violated cut for the remaining scenarios. Bertsimas et. al.  introduced a stochastic scheme for generating Bender cuts, and in this section, we extend it to the context of generating Lagrangian cuts.

\begin{definition}[cut strength]
\label{violation}
For a Lagrangian cut $\theta_s \geq \Bar{Q}_s(\pi, 1) - \pi^{\top} x$, the strength at $\hat{x}$ is given by  
    \begin{align}
        V(\hat{x}, \pi) &= Q(\hat{x}, s) - q_s(\hat{x}, \pi) ,       
      \end{align}  
where $q_s(\hat{x}, \pi)=\Bar{Q}_s(\pi, 1) - \pi^{\top} \hat{x}$ measures the value of $\theta_s$ associated with this cut given the input $\hat{x}$.
\end{definition}

It is obvious that the cut strength $V(\hat{x}, \pi)$ measures the gap between the current Lagrangian cut $(\pi, 1)$ and the most violated cut at $\hat{x}$. \Cref{fig:example} presents an illustration of this definition, in which the black line represents the Lagrangian cut $(\pi, 1)$, and the red vertical line intersects with x-axis at point $\hat{x}$, and we thus have $V(\hat{x}, \pi) = \triangle h$.

\begin{figure}[htbp]
	\centering
	
	\begin{minipage}{0.8\linewidth}
		\vspace{3pt}
		\centerline{\includegraphics[width=\textwidth]{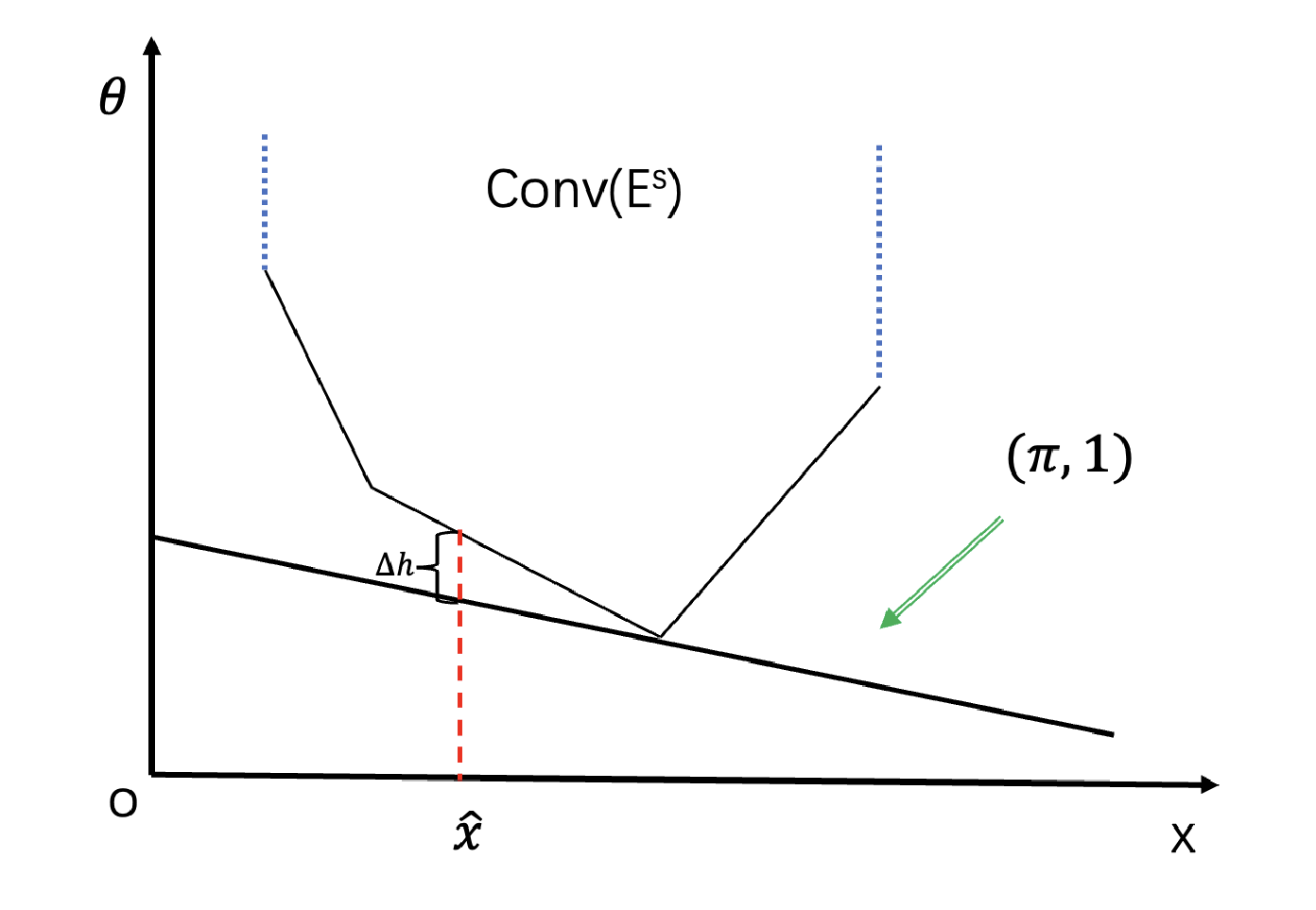}}
	 
		%\centerline{(a)}
	\end{minipage}

	\caption{An illustration for \Cref{violation}}
	\label{fig:example}
\end{figure}

\begin{remark}
The cut strength fucntion $V(\hat{x}, \pi)$ is non-negative and Lipschitz continuous with respect to the variable $\pi$. The reason of the latter is that it is the difference of two Lipschitz continuous functions. Moreover, $V(\hat{x}, \pi_s) = 0$ for each scenario $s$, where $(\pi_s, 1)$ is the most violated Lagrangian cut at $\hat{x}$. 
\end{remark}

We then give the important concept serving for generating new Lagrangian cut. 

\begin{definition}[averaged Lagrangian cut] Given a scenario set $S$, $\forall\Bar{S}\subseteq S$ and $\forall s\in\Bar{S}$, denote the Lagrangian cut about $s$ by $(\pi_s, 1)$, then the cut $(\sum_{s\in \Bar{S}}\pi_s/\lvert \Bar{S} \rvert, 1)$ is called averaged Lagrangian cut.
\end{definition}

The following proposition guarantees the quality of the averaged Lagrangian cut.  

\begin{proposition}
    Fix current relaxed optimal first-stage solution $\hat{x}$. Denote $v^2$ by the variance of the optimal dual solutions, that is $v^2 = \frac{1}{\lvert S \rvert} \sum_{s \in S}\Vert\pi_s - \Bar{\pi}^S \Vert^2$, $\Bar{\pi}^S = \frac{1}{\lvert S \rvert}\sum_{s \in S} \pi_s$, $\Bar{\pi}^{\Bar{S}} = \frac{1}{\lvert \Bar{S} \rvert}\sum_{s \in \Bar{S}} \pi_s$. There exists constants $L, M > 0$, such that for any $\delta \in (0,e^{-1})$, when $\Bar{S}$ is sampled from S without replacement in a fixed size, we have with probability higher than $1 - 3\delta$, such that:
    \begin{align}
        \label{pro1}
        \sum_{s \notin \Bar{S}}V(\hat{x}, \Bar{\pi}^{\Bar{S}}) \leq L\sqrt{\lvert S \backslash   \Bar{S} \rvert}v + D\sqrt{\lvert S\backslash \Bar{S} \rvert \log(1/\delta)}
    \end{align}
where 
    \begin{align}
        D = LM\sqrt{n}[\sqrt{\lvert S\rvert}(\frac{1}{\lvert \Bar{S}\rvert} - \frac{1}{\lvert S \rvert})^{1/2} + (\frac{1}{\lvert S \backslash \Bar{S} \rvert} - \frac{1}{\lvert S \rvert})^{1/4}]
    \end{align}
\end{proposition}
\begin{proof}
The result is essentially an extension for Bender cut, and the proof is similar to that in \cite{bertsimas2023stochastic}.  
\end{proof}

 This proposition gives a probabilistic guarantee for the quality of our proposed averaged Lagrangian cut. By solving part of the Lagrangian subproblems, we can obtain a not bad cut for the others with some probability.
In our implementation, we incorporate the averaged Lagrangian cut into the framework of Algorithm 3.1 at the end of each iteration.

\section{Computational Study}
\label{sec:experiment}
In this section, we are going to conduct experiments on our proposed two algorithmic schemes. In \Cref{exact separation} and \Cref{restricted separation}, we compare the gap closed by generating Lagrangian cut by batch with that closed by \cite{chen2022generating} and \cite{rahmaniani2020benders}. And in \Cref{average}, we show the strength of the averaged Lagrangian cut. In \Cref{optimality}, we display the results for solving these instances to optimality by branch-and-cut method while generating Lagrangian cut at root node.\

Three classes of problems are considered, including the stochastic server location problem (sslp), a variant of the stochastic server location problem (sslpv) and the stochastic multi-commodity flow problem (smcf). The sslp problem \cite{ntaimo2005million} is a two-stage SIP with pure binary first-stage and mixed-binary second-stage variables. In this problem, the decision maker has to choose from $n_1$ sites to allocate servers with cost in the first stage. Then in the second stage, the availability of each client would be observed and every available client must be served at some site also with cost. The objective is to minimize the total cost. The sslpv problem \cite{chen2022generating} is a variant of the sslp problem. We generate the instances of them as \cite{chen2022generating}. The smcf problem \cite{crainic2001bundle} contains pure binary first-stage and continuous second-stage variables, in which the decision maker has to choose some edges with capacity constraint from the node-edge graph to transfer commodity flows. Then in the second stage, the demand of each commodity is available and must be transferred from its original node to the destination node by the chosen edges. We generate the stochastic counterpart of instances r04 as \cite{rahmaniani2018accelerating}. Our test includes 24 instances for the sslp problem, 24 instances for the sslpv problem and 6 instances for the smcf problem. The information of these instances is listed in \Cref{T1}.\

\begin{table}[ht]
\caption{Profiles of the three classes of instances}
\label{T1}
\begin{center}
\begin{tabular}{llllll}
\hline
Instances& $\lvert S \rvert$& $n_1^\S$& $n_2$& $m_1$& $m_2$\\
\hline
sslp(40-50)$^\dag$& [50, 200]$^*$& 40& 2040& 1& 90\\
sslp(30-70)& [50, 200]& 30& 2130& 1& 100\\
sslp(20-100)& [50, 200]& 20& 2020& 1& 120\\
sslp(50-40)& [50, 200]& 50& 2050& 1& 90\\
sslpv(40-50)& [50, 200]& 40& 2040& 1& 90\\
sslpv(30-70)& [50, 200]& 30& 2130& 1& 100\\
sslpv(20-100)& [50, 200]& 20& 2020& 1& 120\\
sslpv(50-40)& [50, 200]& 50& 2050& 1& 90\\
smcf(r04.1-r04.6)& [500]& 60& 600& 1&660 \\
\hline
\end{tabular}
%\end{table}
\begin{tablenotes}
    \footnotesize
        \item $^\dag$The first and second digit represent the number of locations and of customers, respectively; $^*$There are two cases of $|S|=50$ and $|S|=200$; $^\S$$n_1$, $n_2$, $m_1$ and $m_2$ share the same meanings with those given in \cref{equa:SIP_ori}.   
\end{tablenotes}
\end{center}
\end{table}

At the very first, we will elaborate two paradigms for separating Lagrangian cut by exact and by restricted mode \cite{chen2022generating}. For the exact separation, $\Pi_s$ in Algorithm A.1 is chosen to be a neighborhood of the original point at the Euclidean space, while for the restricted separation, $\Pi_s$ is chosen to be a linear subspace whose basis consists of the coefficient vectors of some already generated Bender cuts, and in \cite{chen2022generating}, they propose an MIP to choose the best basis with a preset size $K$. More specific information can refer to \cite{chen2022generating}. Therefore, in our experiment, we separate the Lagrangian cut by the two paradigms respectively to prove efficiency and applicability of our proposed algorithm.\

- \textbf{Exact}: Exact separation of Lagrangian cut.\

- \textbf{RstrMIP}: Restricted separation of Lagrangian cut.\

These two methods are different cut generators. We incorporate them into the line 11 in Algorithm 3.1 respectively to test our batch algorithm in different circumstances. Furthermore, we list the approach that we will adopt in the experiment here:\

- \textbf{Exact-Tra}: Generating Lagrangian cut exactly by no batch (i.e., for every scenario at each iteration.)\

- \textbf{Exact-Lbb($\beta$)}: Generating Lagrangian cut exactly by batch  and the batch size is $\lfloor \lvert S \rvert \beta\rfloor$.\

- \textbf{RstrMIP-Tra}: Generating Lagrangian cut  with restricted separation by no batch.\

- \textbf{RstrMIP-Lbb($\beta$)}: Generating Lagrangian cut with restricted separation by batch and the batch size is $\lfloor \lvert S \rvert \beta \rfloor$.\

\subsection{Implementation Details}
In our experiment, instances of sslp and smcf  are conducted on a Mac laptop with 16GB RAM and an Apple M1 pro processor, while instances of sslpv are conducted on a Mac laptop with 16GB RAM and an Apple M1 processor. All related programs such as ?(LPs), MIPs and ?(QPs) are solved using the optimization solver Gurobi 10.0.3 for all these instances.\

When implementing the process of generating Lagrangian cut in Algorithm A.1 and A.2, we set the circumstances as in that paper to reproduce their results. In their study, they highlighted that the restricted separation algorithm demonstrated optimal performance when $\delta$ was set to $50\%$ and $K$ was set to 10. Hence, in our subsequent experiments, we set $\delta$ to $50\%$ and $K$ to 10 when addressing the separation problem \cref{equa9}.\

 For Line 8 in Algorithm 3.1, we fix the order of these batches to be solved, that is: if we end up current iteration at batch $t$, then we begin the next iteration from batch $t+1$. We generate Lagrangian cut for instances of sslp and sslpv in two paradigms, that is in exact and in restricted paradigm. For the instances of smcf, we only conduct experiments in the exact paradigm. This decision is based on the observation that the two separation methods perform nearly identically in this problem class.

\subsection{Test for generating Lagrangian cut by batch at root node}
In this section, We test the performance of Algorithm 3.1 in improving lower bound of the relaxed master problem at the root node. We set the time limit to one hour for $\Cref{a1}$ in this section.\

\subsubsection{Exact-Tra vs. Exact-Lbb}
\label{exact separation}
We first examine the impact of different batch sizes ($\beta = 5\%, 10\%, 20\%, 50\%$) in the exact separation paradigm. Because the trending for these instances is similar, we only depict the convergence profile and the changing of number of cuts added with the lower bound improving for one sslp instance (sslp1-30-70-200), one sslpv instance (sslpv1-30-70-200) and one smcf instance (r04.3). \cref{trending for sslp} - \cref{trending for smcf} presents the pictures for the three instances respectively. Specifically, taking the sslpv1-30-70-200 as an example, the left picture (a) represents the changing of lower bound over time and the right one (b) represents the changing of number of cuts added over lower bound in the exact separation paradigm. It is evident that the lower bound experiences the most rapid improvement when $\beta = 5\%$ and $\beta = 10\%$. Moreover, the same lower bound value can be achieved by generating significantly fewer Lagrangian cuts, particularly in the latter stages of the entire process. Furthermore, the performance of batch processing, regardless of its size, consistently outperforms the scenario of no batch processing in terms of both time consumption and the number of added cuts. This aligns with our theoretical analysis presented in the previous sections. Regarding the observed turning points in these figures, it can be elucidated that at the initial stage, the linear relaxation of the Bender master problem is exceedingly weak. However, as time progresses, the relaxation becomes tighter, accompanied by an increase in the number of added cuts. This phenomenon renders it challenging for a single cut to further enhance the lower bound.
\begin{figure}[htbp]

	\begin{minipage}{0.49\linewidth}
		\vspace{3pt}
		\centerline{\includegraphics[width=\textwidth]{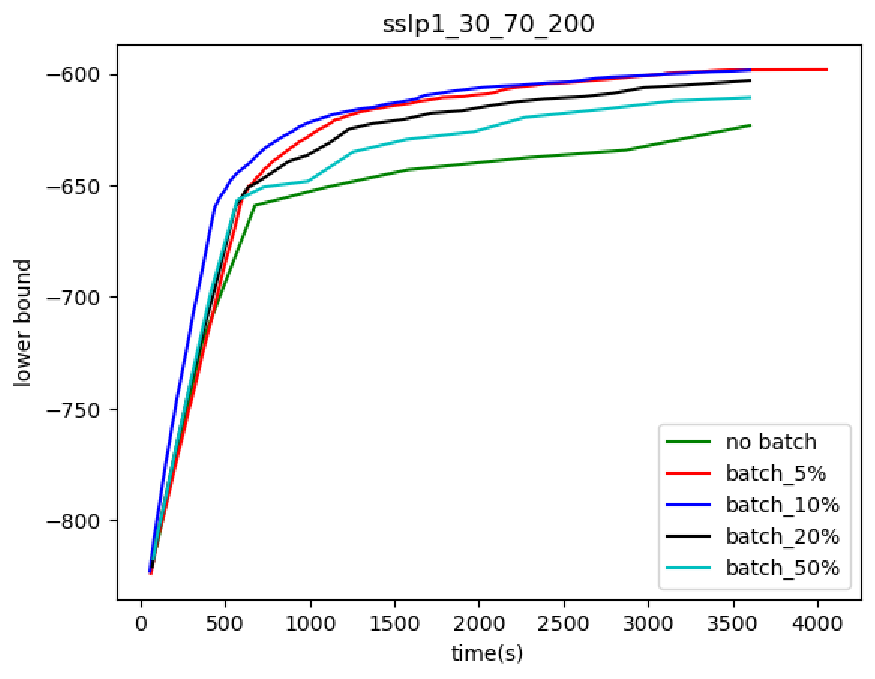}}
	 
		\centerline{(a)}
	\end{minipage}
	\begin{minipage}{0.49\linewidth}
		\vspace{3pt}
		\centerline{\includegraphics[width=\textwidth]{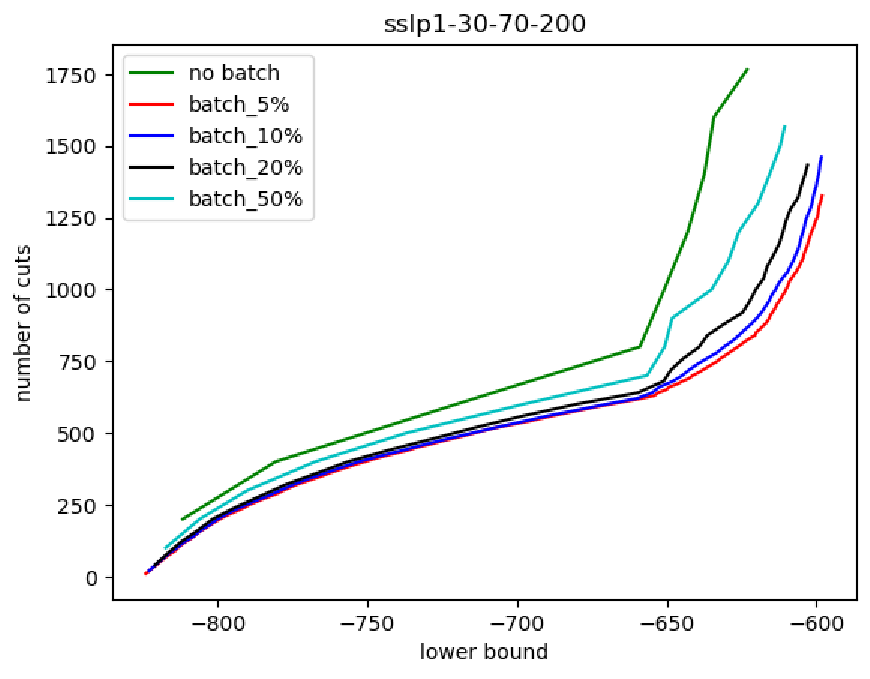}}
	 
		\centerline{(b)}
	\end{minipage}
 
	\caption{Exact-Tra vs. Exact-Lbb($\beta$) on sslp1-30-70-200: (a) lower bound improved; (b) the number of cuts added.}
	\label{trending for sslp}
\end{figure}

\begin{figure}[htbp]

	\begin{minipage}{0.49\linewidth}
		\vspace{3pt}
		\centerline{\includegraphics[width=\textwidth]{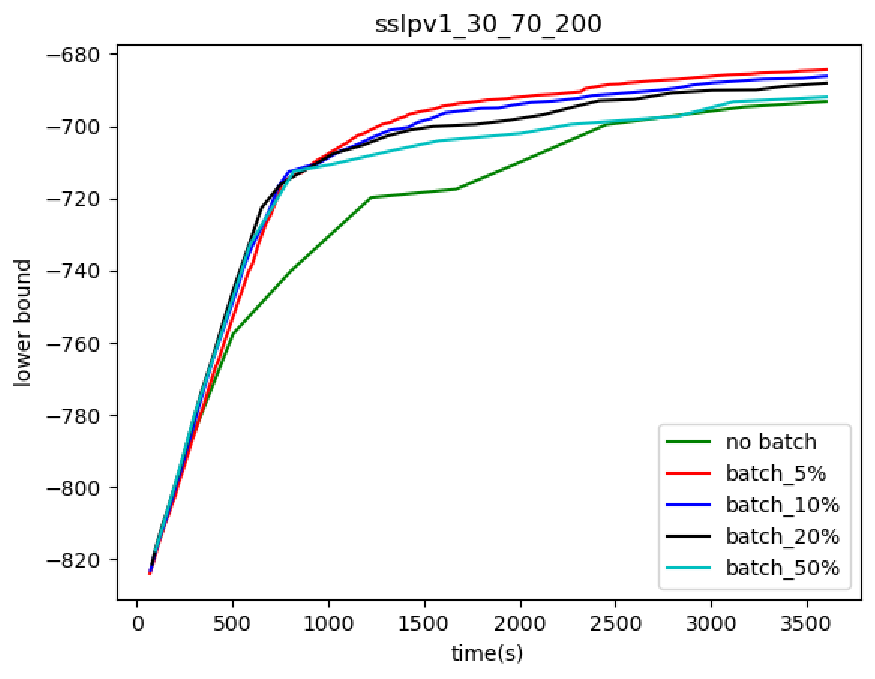}}
	 
		\centerline{(a)}
	\end{minipage}
	\begin{minipage}{0.49\linewidth}
		\vspace{3pt}
		\centerline{\includegraphics[width=\textwidth]{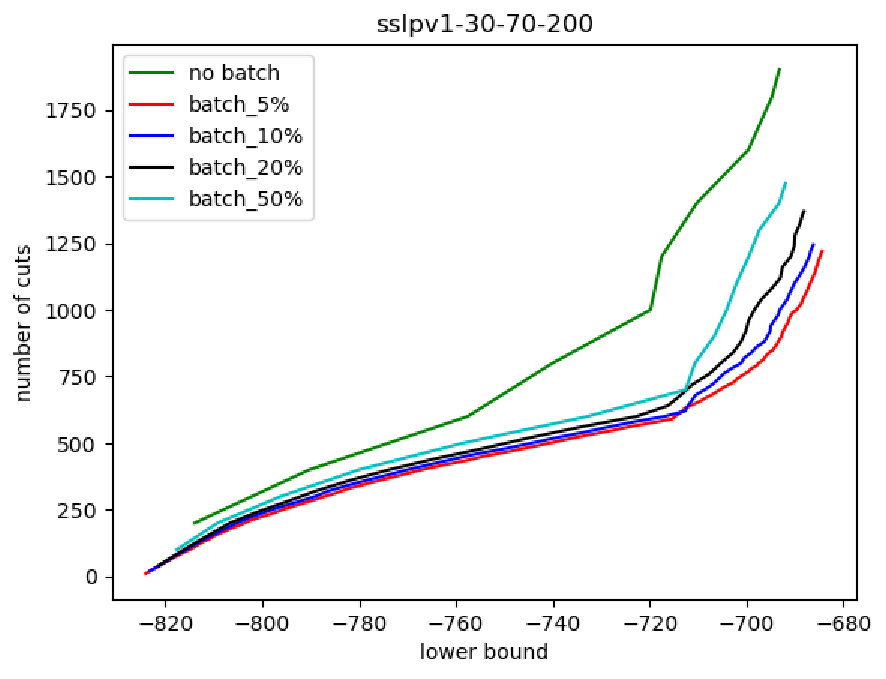}}
	 
		\centerline{(b)}
	\end{minipage}
 
	\caption{Exact-Tra vs. Exact-Lbb($\beta$) on sslpv1-30-70-200: (a) lower bound improved; (b) the number of cuts added.}
	\label{trending for sslpv}
\end{figure}

\begin{figure}[htbp]

	\begin{minipage}{0.49\linewidth}
		\vspace{3pt}
		\centerline{\includegraphics[width=\textwidth]{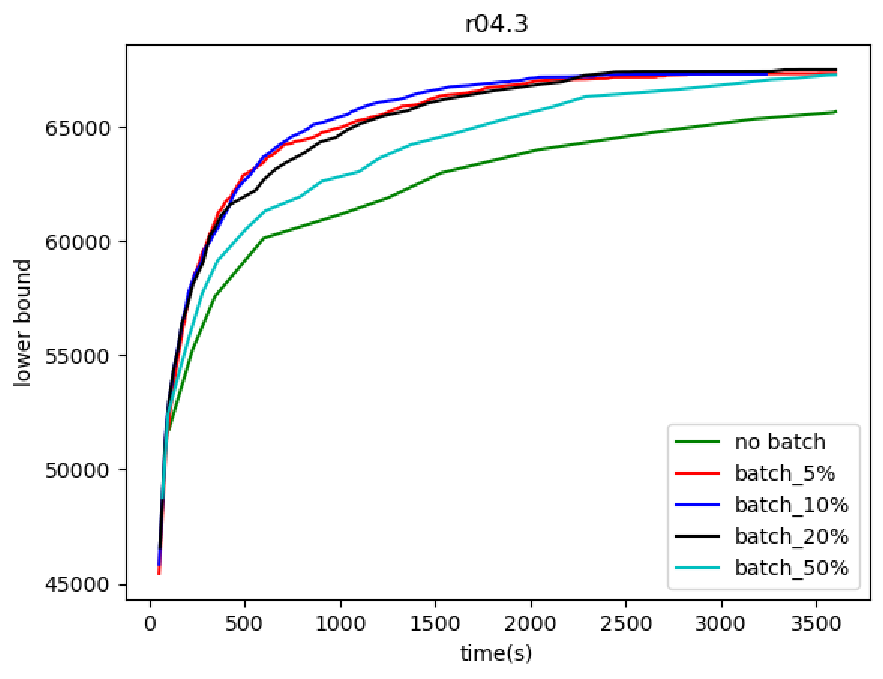}}
	 
		\centerline{(a)}
	\end{minipage}
	\begin{minipage}{0.49\linewidth}
		\vspace{3pt}
		\centerline{\includegraphics[width=\textwidth]{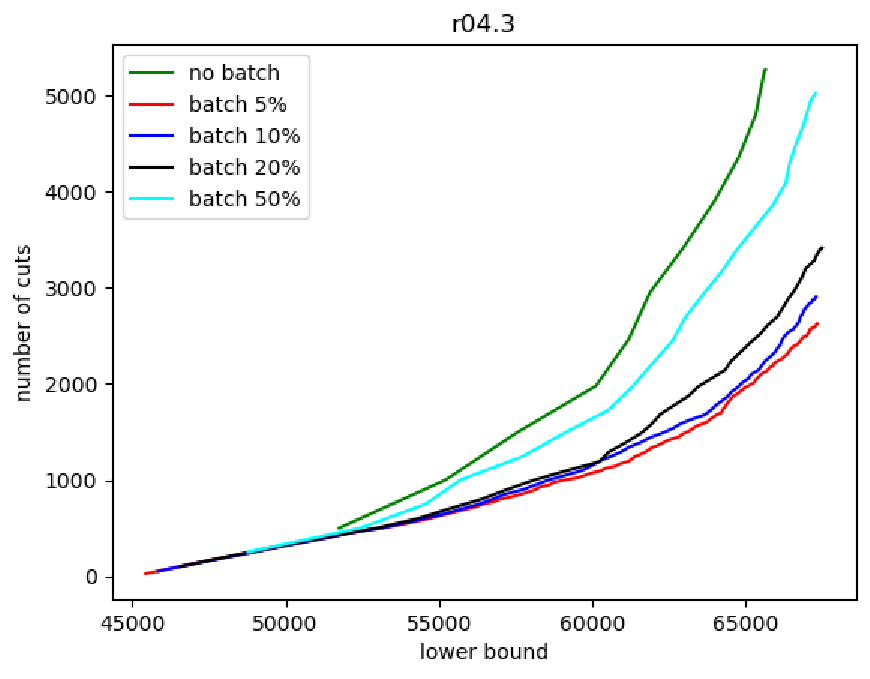}}
	 
		\centerline{(b)}
	\end{minipage}
 
	\caption{Exact-Tra vs. Exact-Lbb($\beta$) on r04.3: (a) lower bound improved; (b) the number of cuts added.}
	\label{trending for smcf}
\end{figure}

In order to show the overall performance of our proposed algorithm across all these instances, we make use of the notion of $\gamma$-gap-colsed profile as in \cite{chen2022generating}: given a set of problem instances $P$ and a set of cut generation methods $M$ (e.g. different batch sizes in our situation), $g_p$ denotes the largest gap closed by any of these method in $M$ for instance $p$. The $\gamma$-gap-closed profile is defined with respect to certain threshold $\gamma \in [0,1]$. Given a method $m$ and instance $p$, we define $t^{\gamma}_{p,m}$ as the earliest time of closing the gap by at least $\gamma g_p$. The $\gamma$-gap-closed profile is a figure representing the cumulative growth (distribution function) of the $\gamma$-gap-closed ratio $\rho^{\gamma}_{m}(\tau)$ over time $\tau$ where
$$
\rho^{\gamma}_{m}(\tau) = \frac{\lvert \left\{p \in P: t^{\gamma}_{p,m} \leq \tau \right\} \rvert}{\lvert P \rvert}.
$$

\begin{figure}[htbp]

	\begin{minipage}{0.49\linewidth}
		\vspace{3pt}
		\centerline{\includegraphics[width=\textwidth]{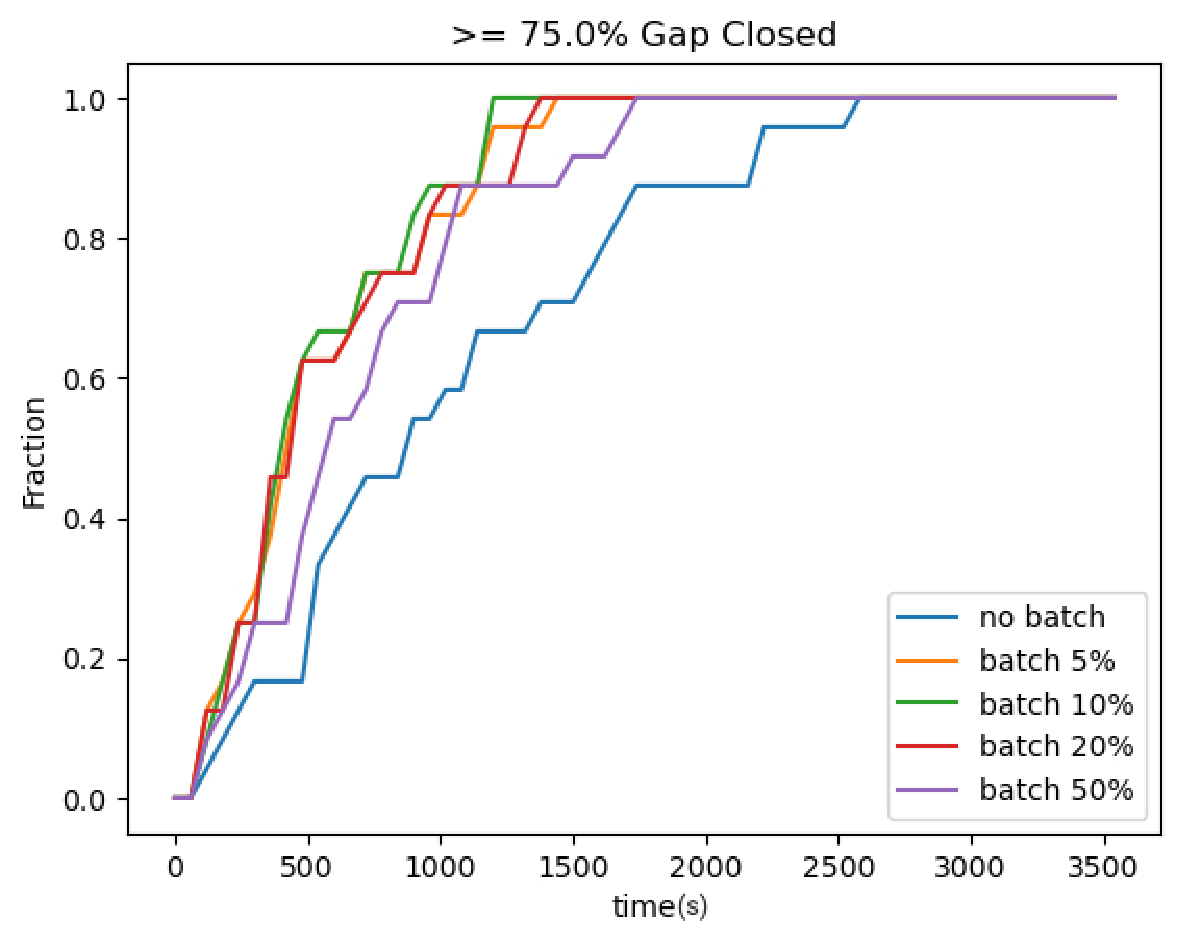}}
	 
		\centerline{(a)}
	\end{minipage}
	\begin{minipage}{0.49\linewidth}
		\vspace{3pt}
		\centerline{\includegraphics[width=\textwidth]{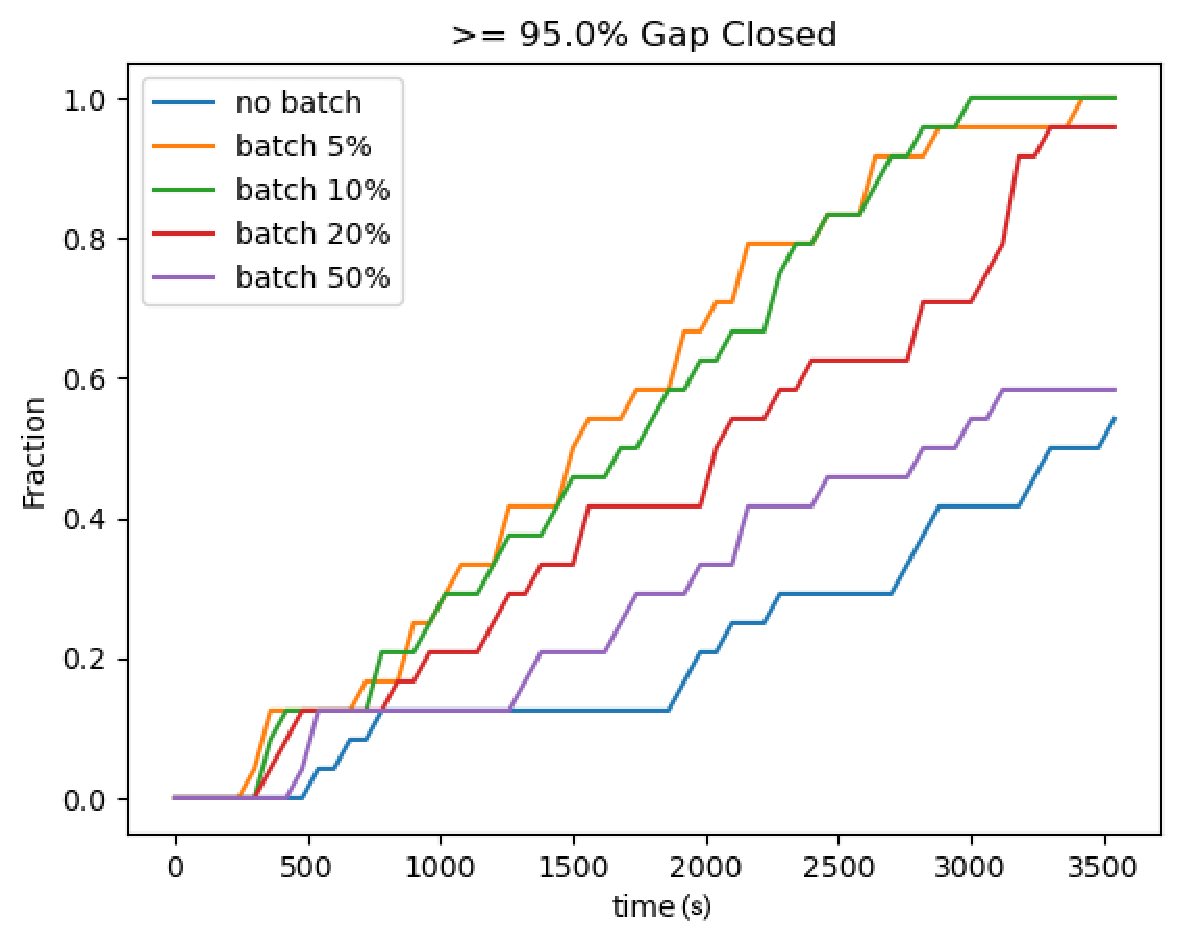}}
	 
		\centerline{(b)}
	\end{minipage}
 
	\caption{$\gamma$-gap-closed profile for sslp instance obtained by Exact separation with $\gamma = 75\%$ (left) and $\gamma = 95\%$ (right)}
	\label{exact_sslp_gap}
\end{figure}

\begin{figure}[htbp]

	\begin{minipage}{0.49\linewidth}
		\vspace{3pt}
		\centerline{\includegraphics[width=\textwidth]{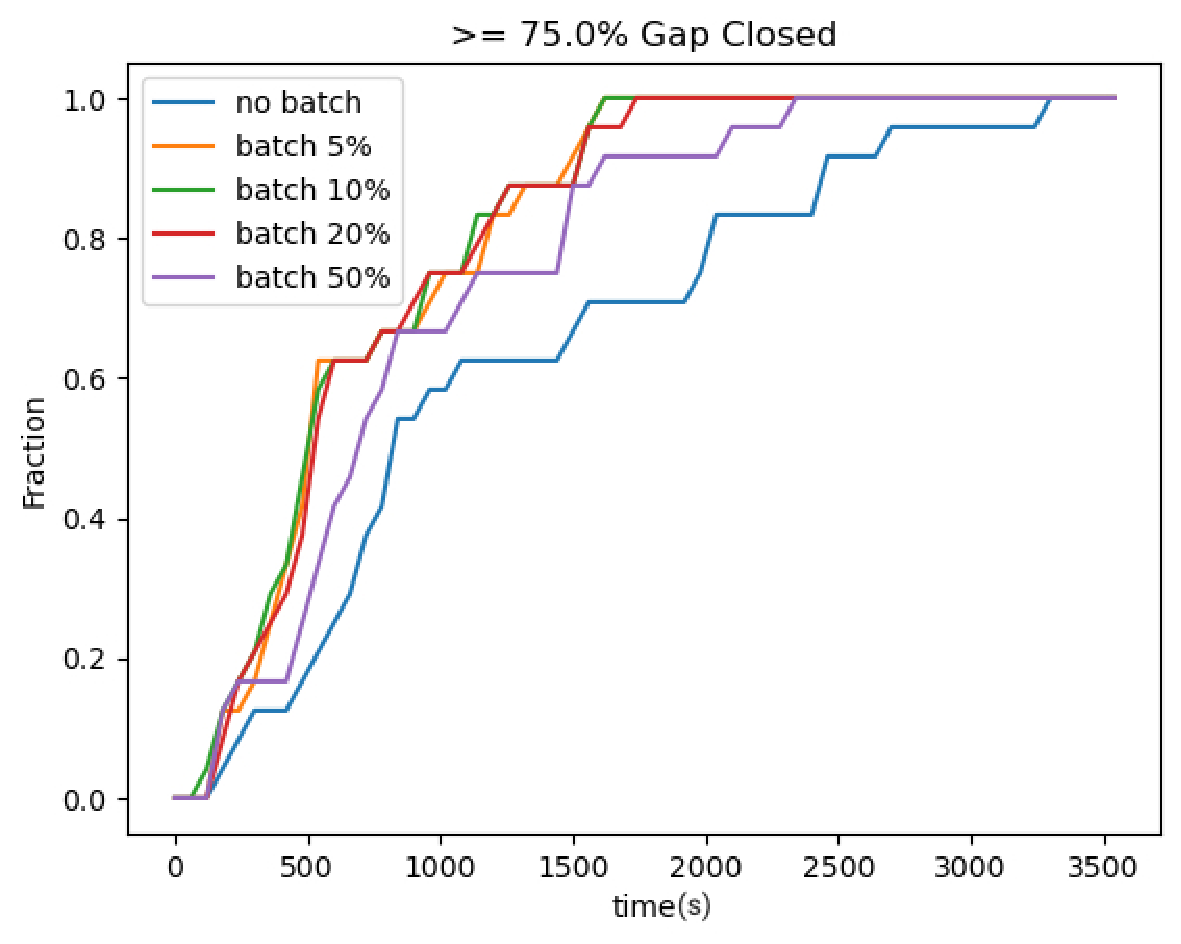}}
	 
		\centerline{(a)}
	\end{minipage}
	\begin{minipage}{0.49\linewidth}
		\vspace{3pt}
		\centerline{\includegraphics[width=\textwidth]{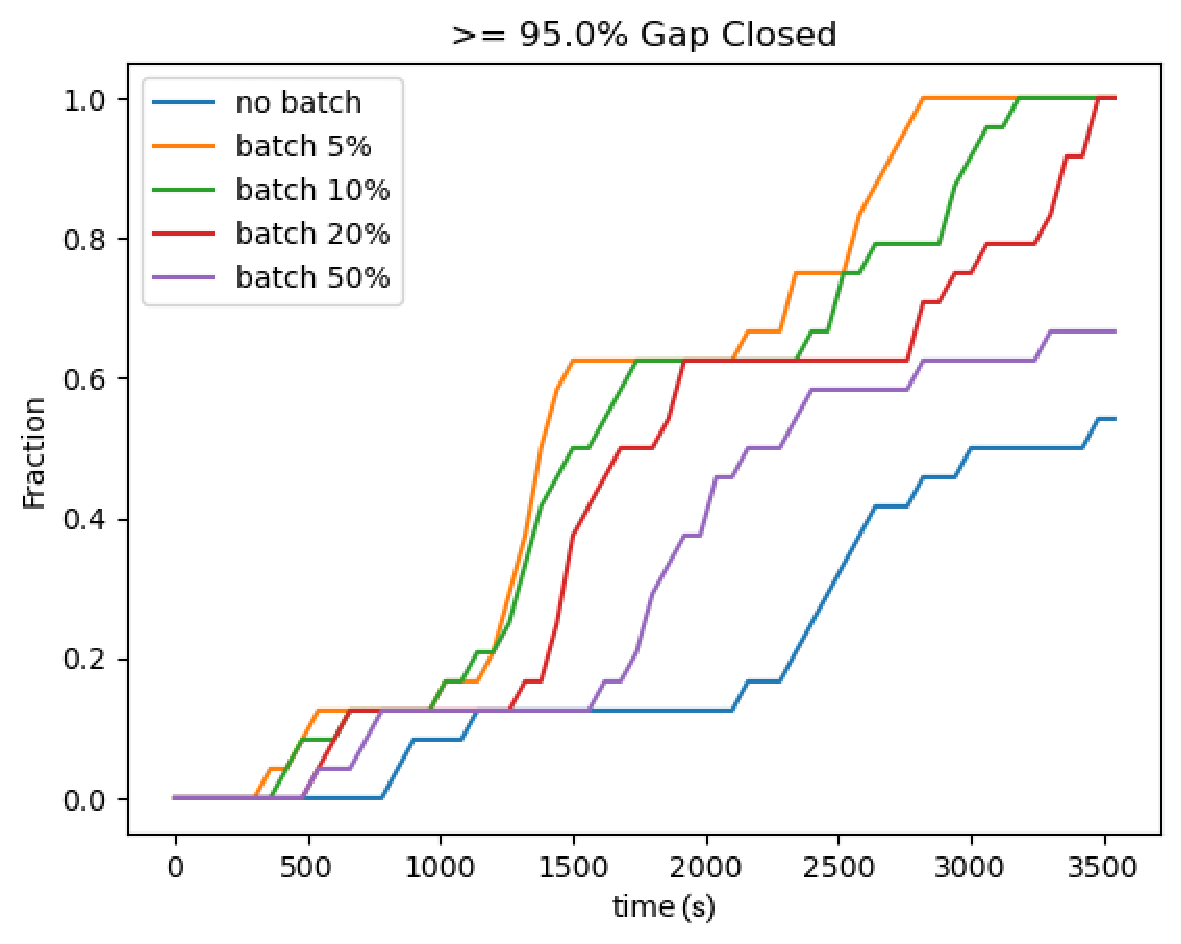}}
	 
		\centerline{(b)}
	\end{minipage}
 
	\caption{$\gamma$-gap-closed profile for sslpv instance obtained by Exact separation with $\gamma = 75\%$ (left) and $\gamma = 95\%$ (right)}
	\label{exact_sslpv_gap}
\end{figure}

\Cref{exact_sslp_gap} and \Cref{exact_sslpv_gap} show the $\gamma$-gap-closed profile in the case of exact separation of the sslp and sslpv instances, respectively. From the two figures we can see that our proposed algorithm dominates the classic one to some extend in the case of exact separation. 
Moreover, the advantage is particularly pronounced in the later stages of the algorithm, as evident in the $95\%$-gap-closed profile in the two figures. Notably, the smaller the batch size, the more superior the performance, aligning with our theoretical results.

\subsubsection{RstrMIP-Tra vs. RstrMIP-Lbb}
\label{restricted separation}
We continue to compare the experimental results between \textbf{RstrMIP-Tra} and \textbf{RstrMIP-Lbb}($\beta$), shown in \Cref{trending for res_sslp} and \Cref{trending for res_sslpv} for sslp1-40-50-200 and sslpv1-40-50-200, respectively. It is clear that our proposed algorithm continues to perform well in the restricted separation paradigm, with the observed situation closely mirroring that in the exact separation paradigm.

%In contrast to the above content showing the result of exact separation, in this subsection, we will present the experimental results where our proposed generating Lagrangian cut  by batch is applied to the second separation paradigm--restricted separation proposed by \cite{chen2022generating}. The structure in this subsection is similar to that in the above one.

%During the execution of restricted separation of Lagrangian cut (Algorithm A.1), rounds of Bender cuts will be generated after each iteration of the separation of Lagrangian cuts. These generated Bender cuts are candidates for the basis of $\Pi_s$. In our algorithm, we will generate Lagrangian cut by batch, which means we can update the set of generated Bender cuts more frequently. And in fact, updating more frequently can make the set of Bender cuts more suitable to the current Bender master problem in theory, therefore generating more qualified Lagrangian cut. However in our pre-experiment, we found that generating Bender cut only after a round of Lagrangian cut for every scenario in $S$ are separated performed the best. Therefore, the result presented in this paper is based on the latter approach.
\begin{figure}[htbp]

	\begin{minipage}{0.49\linewidth}
		\vspace{3pt}
		\centerline{\includegraphics[width=\textwidth]{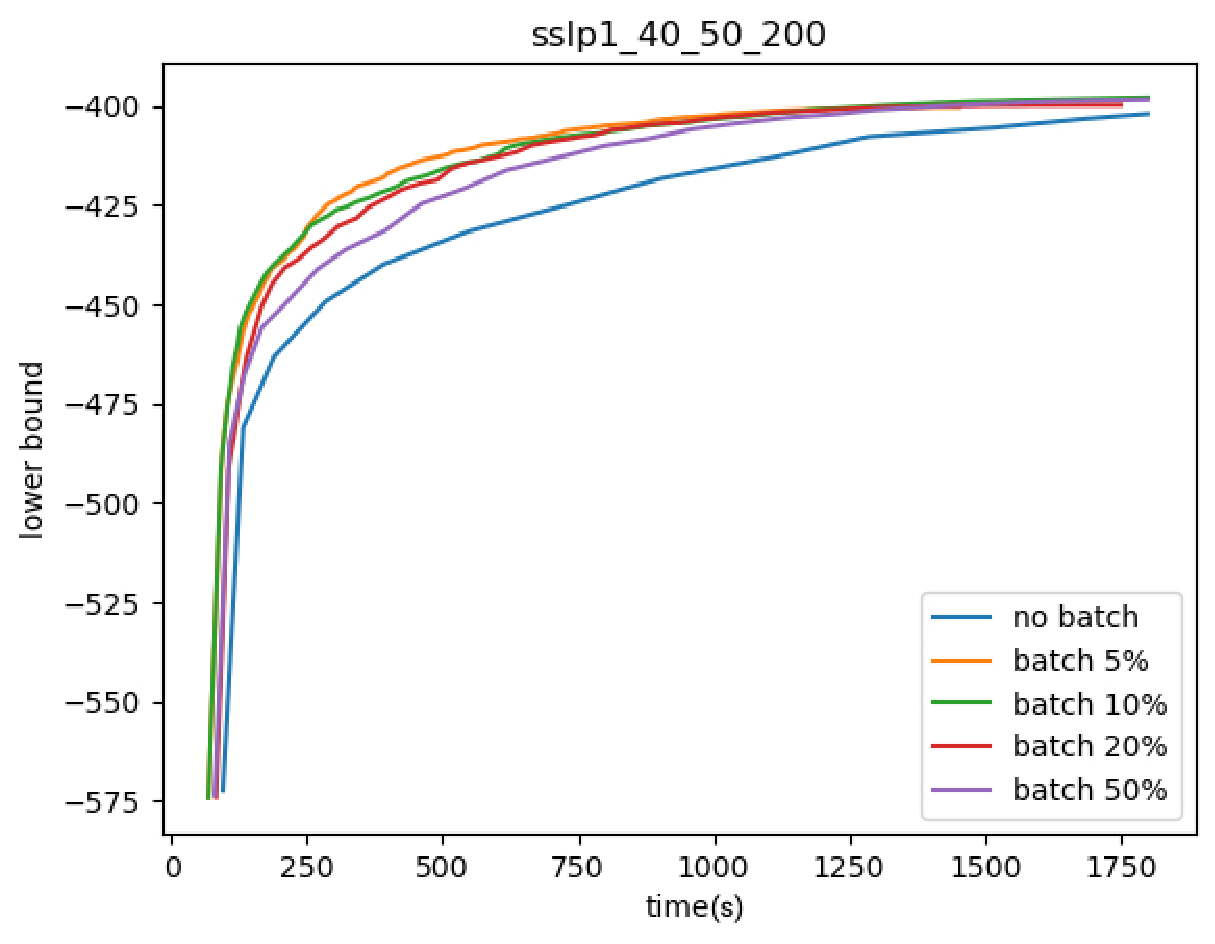}}
	 
		\centerline{(a)}
	\end{minipage}
	\begin{minipage}{0.49\linewidth}
		\vspace{3pt}
		\centerline{\includegraphics[width=\textwidth]{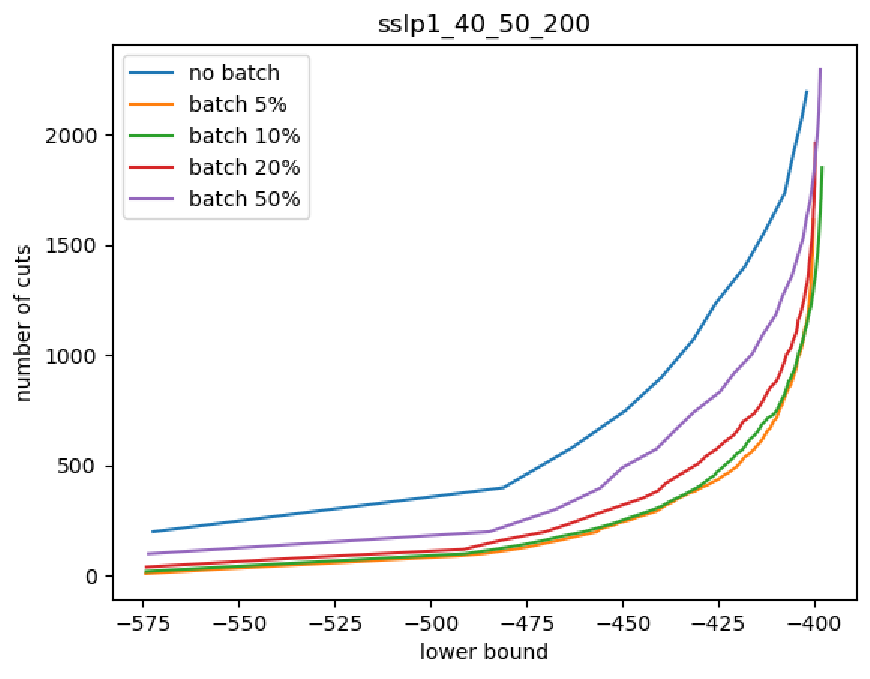}}
	 
		\centerline{(b)}
	\end{minipage}
 
	\caption{RstrMIP-Tra and RstrMIP-Lbb($\beta$) on sslp1-40-50-200: (a) lower bound improved; (b) the number of cuts added.}
	\label{trending for res_sslp}
\end{figure}

\begin{figure}[htbp]

	\begin{minipage}{0.49\linewidth}
		\vspace{3pt}
		\centerline{\includegraphics[width=\textwidth]{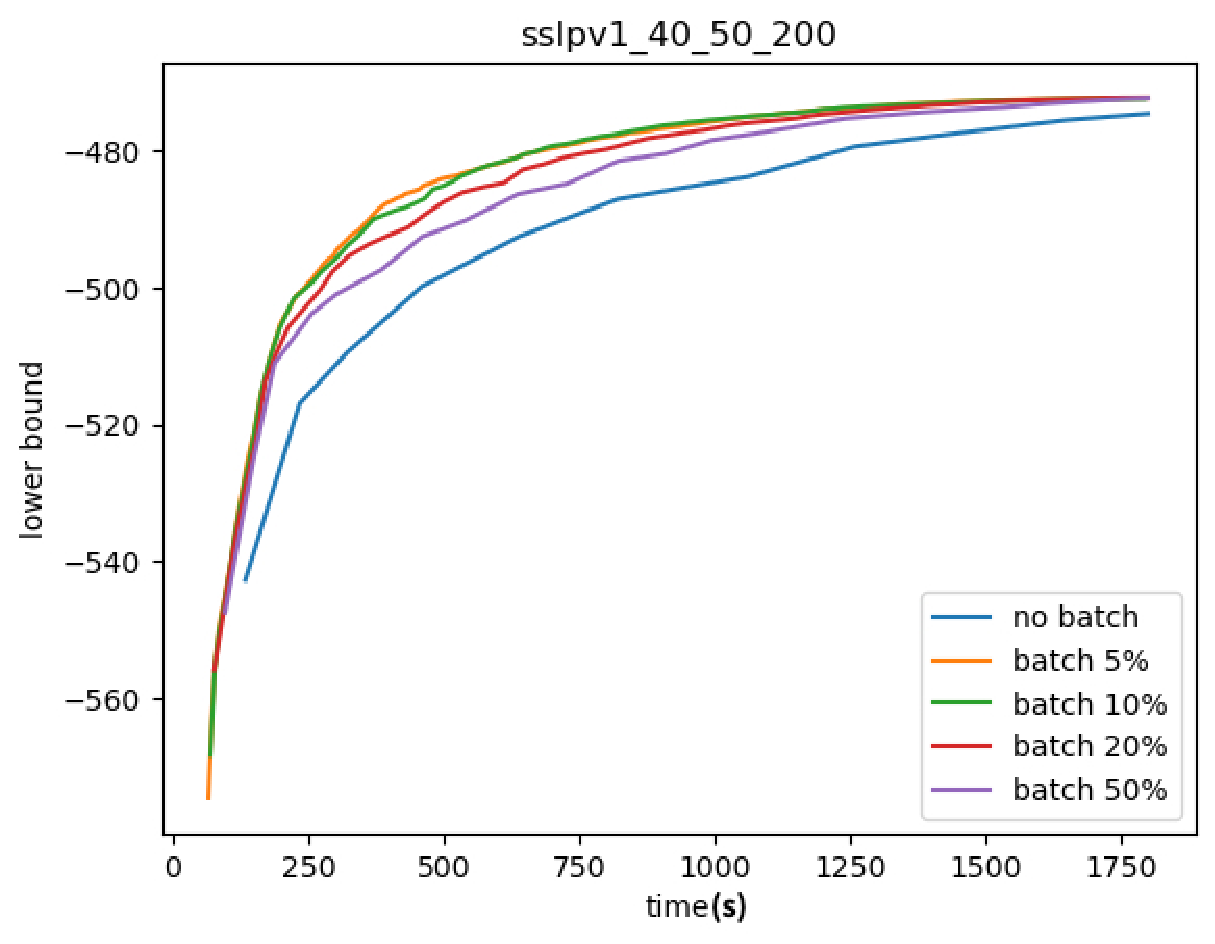}}
	 
		\centerline{(a)}
	\end{minipage}
	\begin{minipage}{0.49\linewidth}
		\vspace{3pt}
		\centerline{\includegraphics[width=\textwidth]{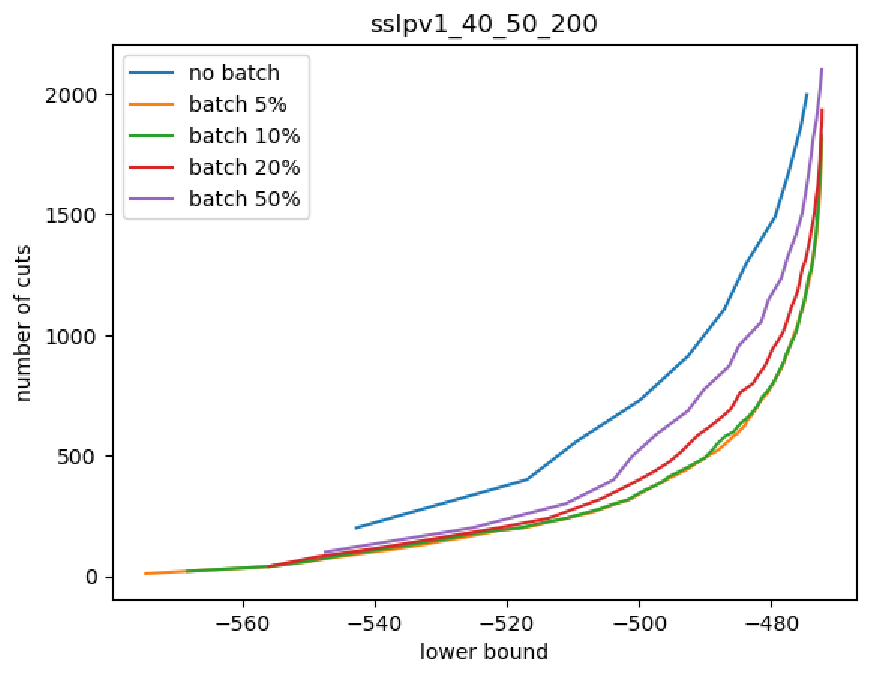}}
	 
		\centerline{(b)}
	\end{minipage}
 
	\caption{RstrMIP-Tra and RstrMIP-Lbb($\beta$) on sslpv1-40-50-200: (a)  lower bound imroved; (b) the number of cuts added.}
	\label{trending for res_sslpv}
\end{figure}

%\Cref{trending for res_sslp} and \Cref{trending for res_sslpv} present the convergence profile (a) and number of cuts added over lower bound (b) in the restricted separation paradigm for sslp1-40-50-200 and sslpv1-40-50-200 respectively. 

%Then, we are going to depict the $\gamma$-gap-closed profile for the two classes of problems (sslp and sslpv) with $\gamma = 0.75$ and $0.95$.
\begin{figure}[htbp]

	\begin{minipage}{0.49\linewidth}
		\vspace{3pt}
		\centerline{\includegraphics[width=\textwidth]{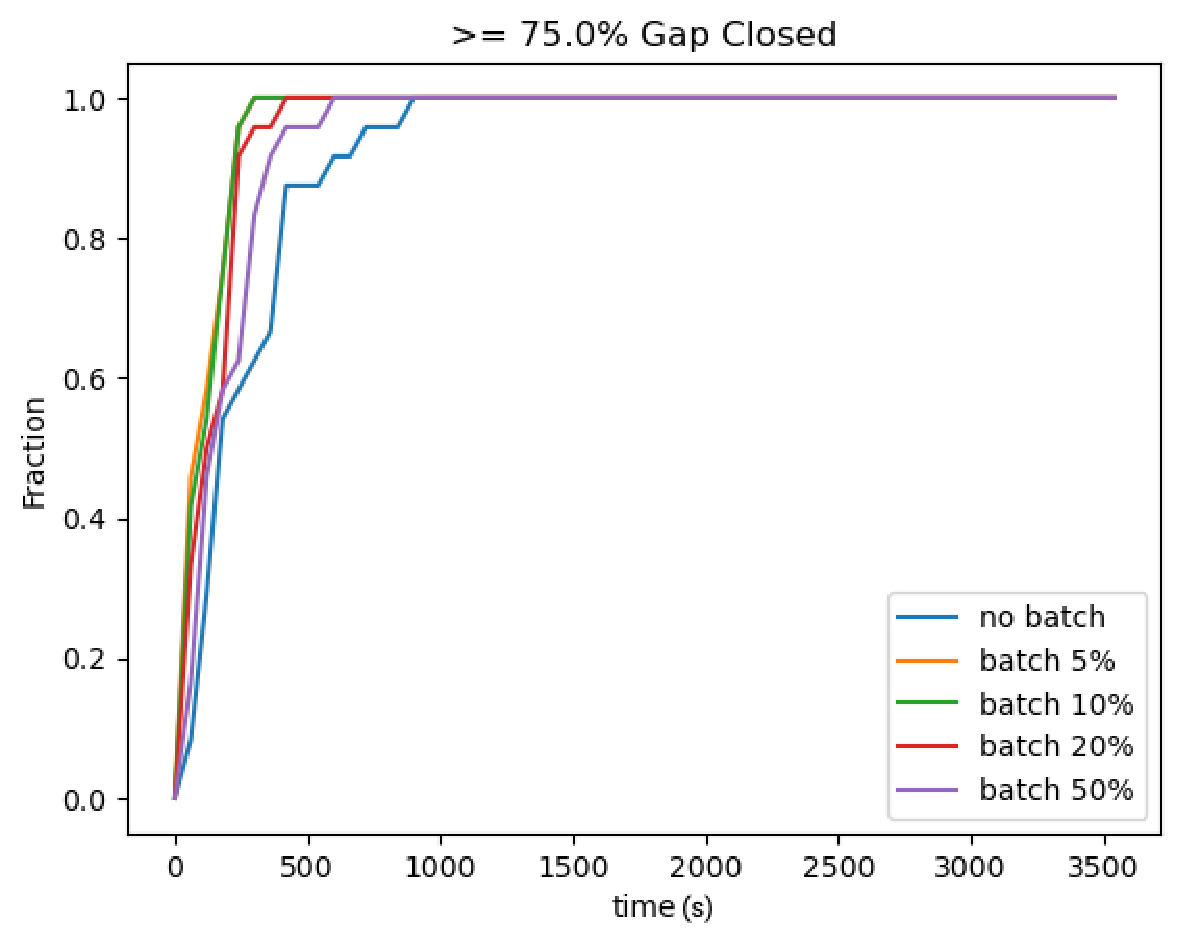}}
	 
		\centerline{(a)}
	\end{minipage}
	\begin{minipage}{0.49\linewidth}
		\vspace{3pt}
		\centerline{\includegraphics[width=\textwidth]{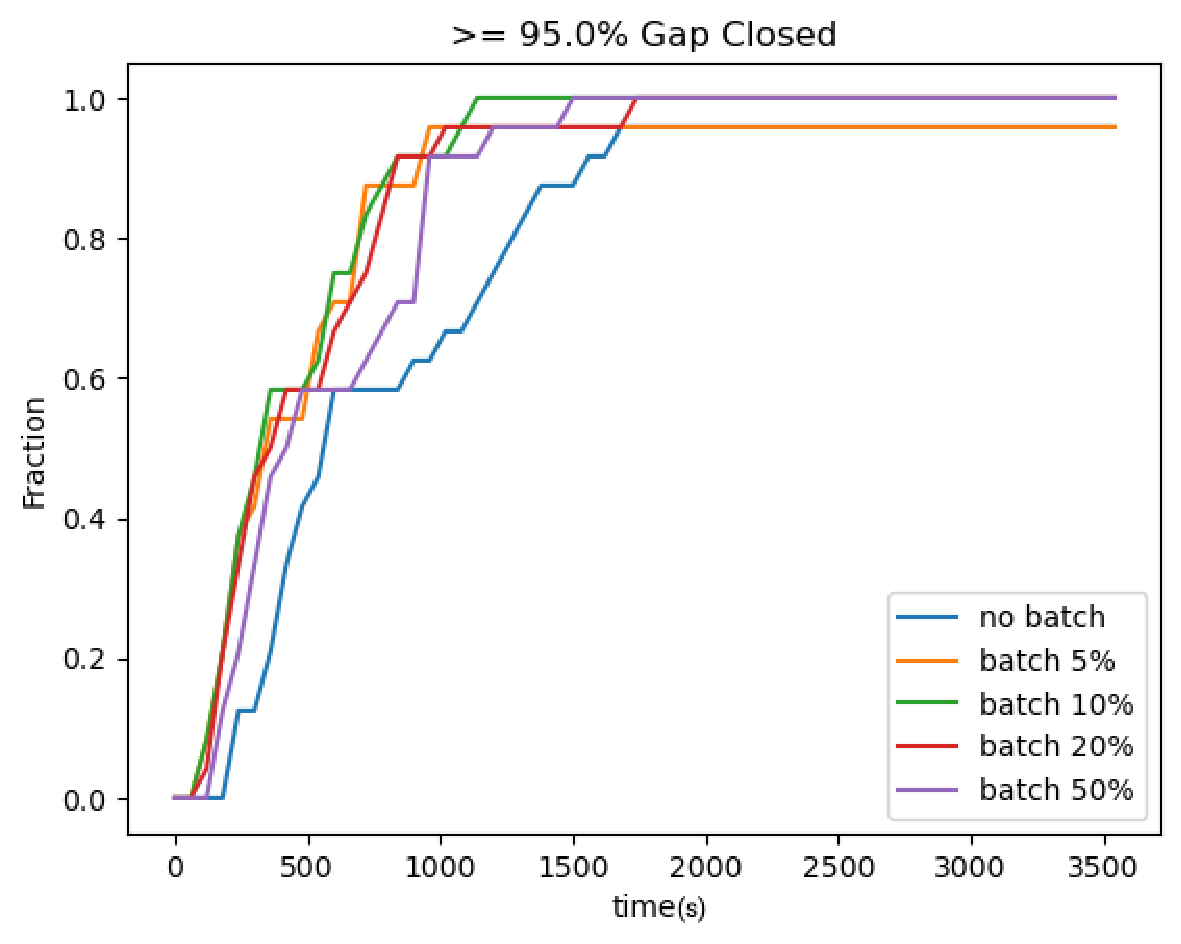}}
	 
		\centerline{(b)}
	\end{minipage}
 
	\caption{$\gamma$-gap-closed profile for sslp instances obtained by restrictive separation with $\gamma = 75\%$ (left) and $\gamma = 95\%$ (right)}
	\label{res_sslp_gap}
\end{figure}

\begin{figure}[htbp]

	\begin{minipage}{0.49\linewidth}
		\vspace{3pt}
		\centerline{\includegraphics[width=\textwidth]{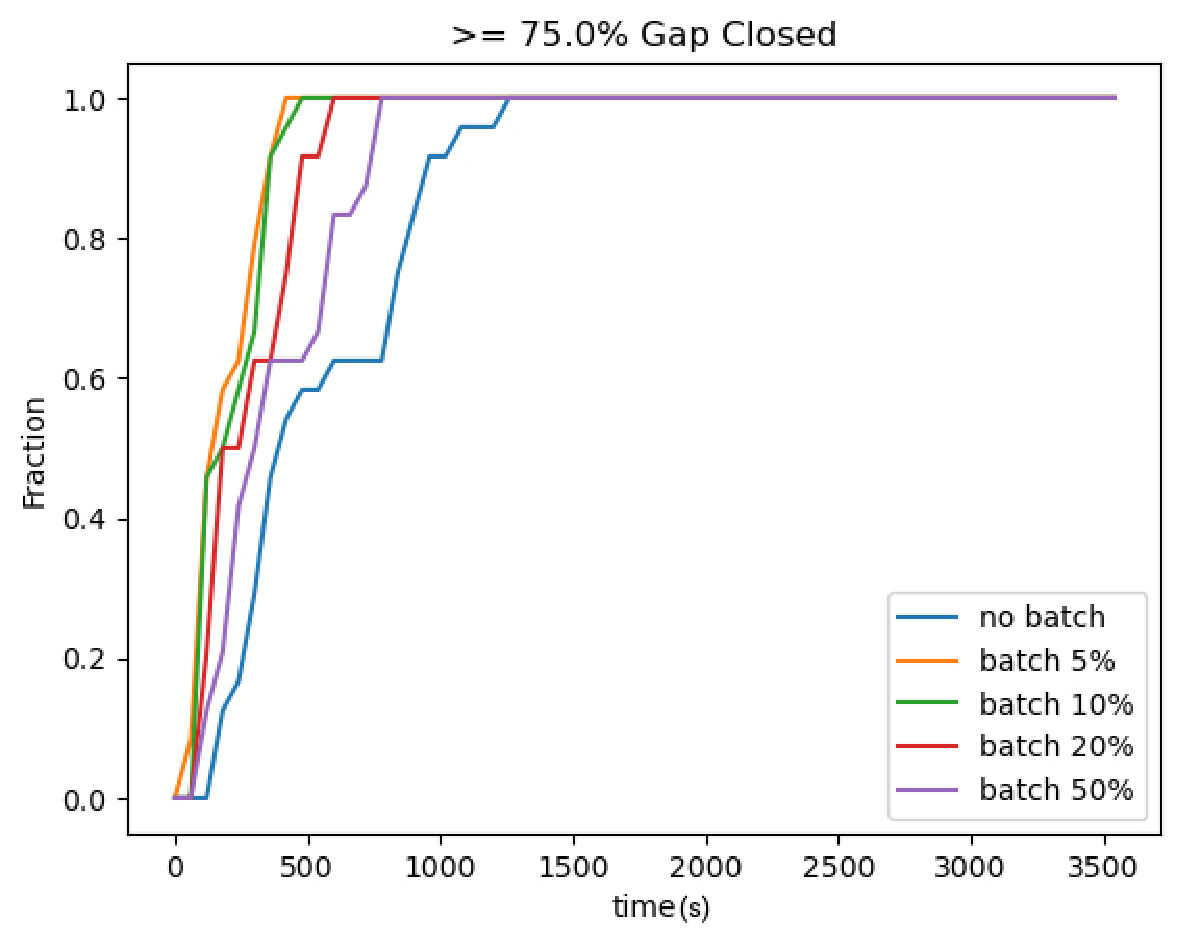}}
	 
		\centerline{(a)}
	\end{minipage}
	\begin{minipage}{0.49\linewidth}
		\vspace{3pt}
		\centerline{\includegraphics[width=\textwidth]{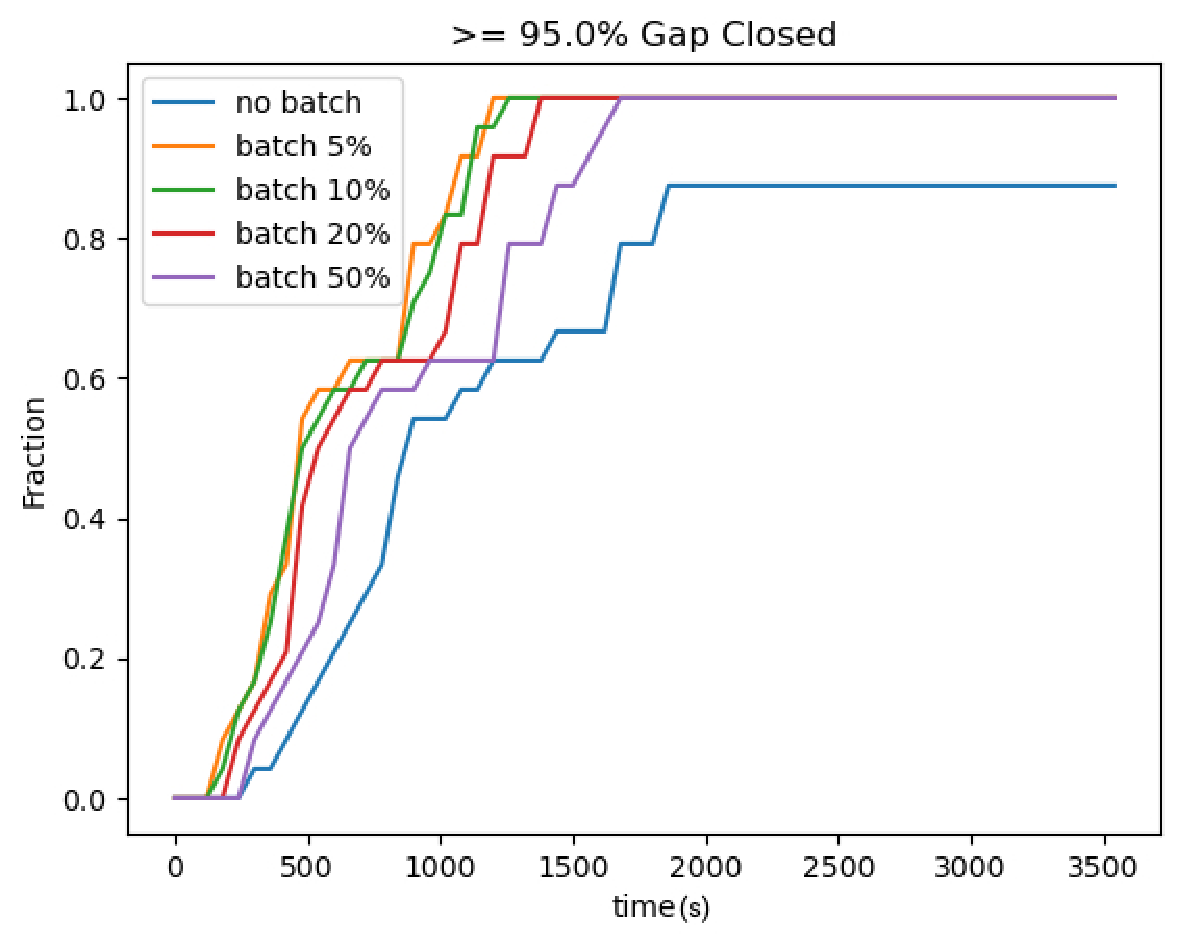}}
	 
		\centerline{(b)}
	\end{minipage}
 
	\caption{$\gamma$-gap-closed profile for sslpv instances obtained by restrictive separation with $\gamma = 75\%$ (left) and $\gamma = 95\%$ (right)}
	\label{res_sslpv_gap}
\end{figure}

Similarly, in the case of restricted separation, we also present the $\gamma$-gap-closed profile of the sslp and sslpv instances in \Cref{res_sslp_gap} and \Cref{res_sslpv_gap}, respectively. We can almost draw the same conclusion as that obtained by exact separation. 
%However, it is noted that in (b) of \Cref{res_sslp_gap}, a too small batch may be not able to achieve the lower bound obtained by no batch (although with little gap). This is a worthwhile direction for our future research: how to update the set of Bender cuts as the basis is the most effective.\

\subsubsection{Quality test of the averaged Lagrangian cut}
\label{average}
In this subsection, we are going to test the quality of the averaged Lagrangian cut proposed in \Cref{sec:average}. When generating Lagrangian cut by batch, at each iteration, Lagrangian cuts of part of these scenarios are generated and  the averaged Lagrangian cut is added for the other scenarios. By this operation, we can save much time spent by solving these Lagrangian subproblems for these scenarios. However, in our numerical experiment, we found that this modification did not perform better than just implementing \Cref{a1}. This may be because the scale of a single scenario of these chosen instances is not too large, and spending time to solve more separation problems to get a stronger Lagrangian cut is worthwhile. Therefore in this subsection, we only present the violation of the averaged Lagrangian cut to show its strength. And it will be an interesting direction to explore the trade-off between the time consumed to solve more separation problems and the strength of the obtained cut for larger scale problems. Given a Lagrangian cut $(\pi_s, 1)$ and the current relaxed optimal solution $(\hat{x}, \hat{\theta}_s)$ for any scenario $s \in S$, we say the violation of the cut at $(\hat{x}, \hat{\theta}_s)$ as $\Bar{Q}_s(\pi_s, 1) - \pi_s \hat{x} - \hat{\theta}_s$.\ %\Cref{percent larger than zero} displays the percentage of 
%Then we will elaborate the violation comparison between the averaged Lagrangian cut and the one obtained by solving the Lagrangian subproblem \cref{equa9}. 

We collected data using the following procedure. We implemented \Cref{a1} with different batch sizes. After each 'while' loop, we computed two violations for each scenario not addressed in that loop. One violation is associated with the averaged Lagrangian cut, while the other corresponds to the Lagrangian cut obtained by solving \cref{equa9}. For each instance, we generated ten full rounds of Lagrangian cuts to collect data. \Cref{percent larger than zero} presents the average percentage of scenarios whose violation of the averaged Lagrangian cut is strictly larger than zero. \Cref{percent violation} displays the average ratio between the two types of violations.

\begin{table}[ht]
\label{percent larger than zero}
\caption{Average percent of scenarios with positive violation}
\centering
\begin{tabular}{llll}
\hline & \multicolumn{3}{c}{ batch size($\beta$) } \\
 Problem class & $5\%$ & $20\%$ & $50\%$\\
\hline 
 sslp & $70\%$& $77\%$  & $83\%$ \\
 sslpv & $64\%$ &  $68\%$&  $63\%$ \\
 smcf & $62\%$ & $66\%$ &  $71\%$ \\
\hline
\end{tabular}
\end{table}

\begin{table}[ht]
\label{percent violation}
\caption{Average ratio between the two violations}
\centering
\begin{tabular}{llll}
 \hline & \multicolumn{3}{c}{ batch size($\beta$) } \\
 Problem class & $5\%$ & $20\%$ & $50\%$\\
\hline
 sslp &$75\%$ &  $80\%$ & $84\%$ \\
 sslpv & $69\%$ &  $70\%$& $75\%$  \\
 smcf & $68\%$ & $77\%$ &  $74\%$ \\
\hline
\end{tabular}
\end{table}

The two tables above indicate that the quality of the averaged Lagrangian cut is commendable. As shown in \Cref{percent larger than zero}, the relaxed optimal solution is cut off by the averaged Lagrangian cut in the majority of scenarios. \Cref{percent violation} demonstrates that the violation of the averaged Lagrangian cut can achieve a high percentage compared to the violation of the Lagrangian cut. Moreover, the performance of the averaged Lagrangian cut generally improves with an increasing batch size, consistent with the observation in proposition 4.5.

\subsection{Results when incorporated into  Branch-and-cut framework}
\label{optimality}
As highlighted in \cite{rahmaniani2020benders, chen2022generating}, the incorporation of Lagrangian cuts extends the processing time at the root node to diminish the size of the branch-and-cut tree. Consequently, in alignment with this perspective, this subsection delves into assessing the performance of achieving optimality through the branch-and-cut method after addressing the root node with \Cref{a1}. Because the above subsection has shown the advantage of the small batches, in this subsection we only consider generating Lagrangian cut by batch $5\%$. We present the result here for smcf with exact separation and sslp and sslpv with restricted separation. Our baseline is the result obtained without batch processing ($\Cref{a0}$). When implementing the algorithm without batch processing, we set a time limit of two hours. The generation of Lagrangian cuts is halted when the gap closed by the last five iterations does not exceed $1\%$ of the total gap closed thus far. Subsequently, the program is integrated into the branch-and-cut framework. For our batch algorithm, the generation of Lagrangian cuts is terminated either when the lower bound matches that achieved by $\Cref{a0}$ or when no further Lagrangian cuts can be generated. This allows for a comparison of the branch-and-cut performance under the same lower bound at the root node. Because  the authors of \cite{rahmaniani2020benders,chen2022generating} have proven the advantage of Lagrangian cut over general branch-and-Bender-cut algorithm, we do not present results for the general branch-and-Bender-cut algorithm here, but only compare the results resulting from generating Lagrangian cut by batch and by no batch.\

\begin{table}
\label{table sslp}
\caption{Comparison of algorithms for solving sslp instances to optimality}
\resizebox{\textwidth}{!}{%
\begin{tabular}{lllll}
\hline & \multicolumn{2}{c}{ On $|S|=50$ instances } & \multicolumn{2}{c}{ On $|S|=200$ instances } \\
 & RstrMIP-Tra & RstrMIP-Lbb($5\%$)  & RstrMIP-Tra & RstrMIP-Lbb($5\%$)\\
\hline \# solved & $12 / 12$  & $12 / 12$ & $12 / 12$ & $12 / 12$  \\
 Avg soln time & 1231 &  832 &3120 & 2232  \\
 Avg gap (\%) & 0.0 & 0.0 & 0.0 & 0.0  \\
 Avg B\&C time & 18 & 23 & 118 & 125  \\
Avg \# nodes & 720 & 758 & 1327 &  1356 \\
\hline
\end{tabular}%
}
\end{table}

\begin{table}
\label{table sslpv}
\caption{Comparison of algorithms for solving sslpv instances to optimality}
\resizebox{\textwidth}{!}{%
\begin{tabular}{lllll}
\hline & \multicolumn{2}{c}{ On $|S|=50$ instances } & \multicolumn{2}{c}{ On $|S|=200$ instances } \\
 & RstrMIP-Tra & RstrMIP-Lbb($5\%$)  & RstrMIP-Tra & RstrMIP-Lbb($5\%$)\\
\hline \# solved & $12 / 12$  & $12 / 12$ & $8 / 12$ & $9 / 12$  \\
 Avg soln time & 2621 &  2015 &4603& 3598  \\
 Avg gap (\%) & 0.0 & 0.0 & 0.35 & 0.27  \\
 Avg B\&C time & 631 & 643 & 1249 & 1224  \\
Avg \# nodes & 7891 & 7810 & 8320 &  8213 \\
\hline
\end{tabular}%
}
\end{table}

\begin{table}[ht]
\label{table smcf}
\centering
\caption{Comparison of algorithms for solving smcf(r04) instances to optimality}
\begin{tabular}{lll}
\hline & \multicolumn{2}{c}{ On $|S|=500$ instances } \\
 & Exact-Tra & Exact-Lbb($5\%$)  \\
\hline \# solved & $6/ 6$  & $6 / 6$ \\
 Avg soln time & 4663 &   3275  \\
 Avg gap (\%) & 0.0 & 0.0  \\
 Avg B\&C time & 348 & 233  \\
Avg \# nodes & 672 & 566 \\
\hline
\end{tabular}
\end{table}
As expected, because the gap at the root node has been greatly reduced by Lagrangian cut, every instance in sslp and smcf can be solved to optimality within a reasonable time consumed for the process of branch-and-cut for all these instances no matter which algorithm we use, by batch or by no batch. While for instances of sslpv, even with the inclusion of Lagrangian cut, the branch-and-cut process is still time-consuming. But our algorithm still perform better than the baseline. We show the computational results in the three tables (\cref{table sslp} - \cref{table smcf}). The three tables contains the information about solving these instances to optimality, respectively for sslp, sslpv and smcf. The information includes the number of instances solved ($\#$ solved), average time when solving to optimality (Avg soln time), average gap between the incumbent solution and the best lower bound (Avg B$\&$C time) and the number of nodes explored during the branch-and-cut process (Avg $\#$ nodes). The tables indicate that generating Lagrangian cuts by batch can effectively diminish the scale of the branch-and-cut tree, particularly when there is a substantial improvement in the lower bound. In essence, the time allocated for exploring the branch-and-cut tree is predominantly influenced by the lower bound at the root node. The predominant time savings occur during the enhancement of the lower bound at the root node. This outcome reaffirms the advantages offered by our proposed algorithm.

\section{Conclusion and points of future study}
\label{conclusion}
We propose to generate Lagrangian cut by batch-a new style for generating Lagrangian cut. We have concluded theoretic analysis for this algorithm, including convergence properties in different situations and lower bound improvement property. Specifically, we provide theoretical proof establishing the advantages of our proposed algorithm. Through extensive experiments on three classes of two-stage Stochastic Integer Programming problems, we demonstrate that our algorithm accelerates lower bound improvement significantly while requiring fewer Lagrangian cuts in both separation paradigms. Consequently, it significantly reduces the time required to achieve optimality when solving instances.\

Moreover, the numerical results affirm the effectiveness of our proposed averaged Lagrangian cut. As a future direction, exploring the utility of the averaged Lagrangian cut in larger-scale instances, where solving a separation problem can be considerably more time-consuming, presents an interesting avenue for investigation. This approach can be conceptualized as a basic learning process. Therefore, we are exploring the prospect of learning to generate Lagrangian cuts in future research. Additionally, investigating the application of batch algorithms in general cutting plane methods poses an intriguing and valuable direction for further exploration.

\appendix
\section{A description about the restricted separation algorithm for generating Lagrangian cut}
\label{appendix}
Diverging from the exact separation method, where $\Pi_s$ is selected as a neighborhood around the original point in Euclidean space, they constrain the feasible region to a simpler space, such as the spanned space by a specific class of Bender cuts. We will show the restricted separation algorithm of generating Lagrangian cut here in Algorithm A.1. 
\begin{algorithm}
\caption{Restricted separation of Lagrangian cut \cite{chen2022generating}}
\begin{algorithmic}[1]\label{a0}
\STATE Initialize $k \leftarrow 0$, $cutfound \leftarrow True$
\WHILE{$cutfound = True$}
\STATE $cutfound \leftarrow False$;
\STATE Solve $\min_{x, \theta_s}\left\{c^{\top}x + \sum_{s \in S}p_s\theta_s :\theta_s \geq \hat{Q}^k_s, Ax = b\right\}$  to obtain current relaxed optimal solution $(x^k, \left\{\theta^k_s\right\}_{s \in S})$
\FOR{$s \in S$}
\STATE Solve Bender subproblem \cref{sub} and obtain dual solution $\lambda^k_s$;
\IF{the generated Bender cut is violated}
\STATE update $\hat{Q}^k_s(x)$ to obtain $\hat{Q}^{k+1}_s(x)$.
\STATE $cutfound \leftarrow True$
\ENDIF
\ENDFOR

\IF{$cutfound = False$}
\FOR{$s \in S$}
\STATE Generate $\Pi^k_s \subset \mathbb{R}^n$
\STATE Solve the separation problem \cref{equa9} with $(\hat{x}, \hat{\theta}_s) = (x^k, \theta^k_s)$ and $\Pi_s = \Pi^k_s$
\IF{the obtained Lagrangian cut is violated}
\STATE update $\hat{Q}^k_s(x)$ to obtain $\hat{Q}^{k+1}_s(x)$.
\STATE $cutfound \leftarrow True$
\ENDIF

\ENDFOR
\ENDIF
\STATE $k \leftarrow k+1$

\ENDWHILE
\end{algorithmic}
\end{algorithm}

Note that the crux of this algorithm resides in solving the separation problem defined by \cref{equa9}. This problem is a bi-level mixed-integer program and is hard to solve. The conventional strategy for addressing this issue involves employing the cutting plane method, wherein $\Bar{Q}_s(\pi, \pi_0)$ is substituted with its relaxation and progressively refined through successive tightening steps. That is $\Bar{Q}_s$ is approximated by a cutting plane model:  
$$
Q_s(\pi, \pi_0) = \min_{x, \theta_s}\left\{\pi^{\top} x +  \pi_0 \theta_s, (x, \theta_s) \in \hat{E}^s\right\}
$$
where $\hat{E}^s$ is a finite subset of $E^s$.
The specific operation can be explained in Algorithm A.2.\

Be aware that the parameter $\delta$ gauges the precision with which the separation problem defined by  \cref{equa9} is resolved, and its significance is evident in influencing the effectiveness of the Lagrangian cut, as demonstrated by the experimental findings in \cite{chen2022generating}. If separating Lagrangian cut exactly, too much time would be consumed. We will elaborate the choice of the parameter when implementing our experiments in \cref{sec:experiment}.

\begin{algorithm}
\label{cutting plane}
\caption{Solution of separation problem \cref{equa9}\cite{chen2022generating}}
\begin{algorithmic}[1]
\STATE Input: $(\hat{x}, \hat{\theta}_s)$, $\Pi_s$, $\hat{E}^s, \delta > 0$
\STATE Output: $(\pi^*, \pi_0^*)$, $\hat{E}^s$
\STATE Initialize $UB \leftarrow + \infty$, $LB \leftarrow - \infty$
\WHILE{$UB > 0$ and $UB - LB \geq \delta UB$}
\STATE $UB \leftarrow \max_{\pi, \pi_0}\left\{Q_s(\pi, \pi_0) - \pi^{\top} \hat{x} - \pi_0 \hat{\theta}_s: (\pi, \pi_0) \in \Pi_s \right\}$, and collect solution $(\pi, \pi_0) = (\hat{\pi}, \hat{\pi_0})$
\STATE Solve \cref{lagsub} to evaluate $\Bar{Q}_s(\hat{\pi}, \hat{\pi}_0)$ and update $\hat{E}^s$ and $Q_s$ with optimal and suboptimal solutions obtained when solving \cref{lagsub}
\IF{$LB < \Bar{Q}_s(\hat{\pi}, \hat{\pi_0}) - \hat{\pi}^{\top} \hat{x} - \hat{\pi}_0 \hat{\theta_s}$}
\STATE $LB \leftarrow \Bar{Q}_s(\hat{\pi}, \hat{\pi_0}) - \hat{\pi}^{\top} \hat{x} - \hat{\pi}_0 \hat{\theta_s}$
\STATE $(\pi^*, \pi_0^*) \leftarrow (\hat{\pi}, \hat{\pi_0})$
\ENDIF
\ENDWHILE
\end{algorithmic}

\end{algorithm}

\section*{Acknowledgments}
This work was funded by the National Nature Science Foundation of China under Grant No. 12320101001 and 12071428.

\bibliographystyle{siamplain}
\bibliography{references}

\end{document}